\documentclass[reqno, 11pt, a4paper, oneside]{amsart}
\usepackage{amsmath}
\usepackage{amsfonts}
\usepackage{amssymb}
\usepackage[ansinew]{inputenc}
\usepackage{graphicx}
\usepackage{amsbsy}
\usepackage{fancyhdr}
\usepackage{yfonts}
\usepackage{hyperref}
\usepackage{enumerate}

\numberwithin{equation}{section}
\newtheorem{teo}{Theorem}[section]
\newtheorem{prop}[teo]{Proposition}
\newtheorem{lema}[teo]{Lemma}

\newtheorem{que}{Question}
\newtheorem{coro}[teo]{Corollary}
\theoremstyle{definition}
\newtheorem{defi}[teo]{Definition}

\theoremstyle{remark}
\newtheorem*{rmk}{Remark}
\newtheorem*{ack}{Acknowledgements}
\sffamily
\title{Concentration estimates for algebraic intersections}
\author{Miguel N. Walsh}
\address{Departamento de Matemática, Facultad de Ciencias Exactas y Naturales, Universidad de Buenos Aires, 1428 Buenos Aires, Argentina}
\email{mwalsh@dm.uba.ar}

\begin{document}

\def\F{\mathbb{F}}
\def\Fqn{\mathbb{F}_q^n}
\def\Fq{\mathbb{F}_q}
\def\Fp{\mathbb{F}_p}
\def\Di{\mathbb{D}}
\def\E{\mathbb{E}}
\def\Z{\mathbb{Z}}
\def\Q{\mathbb{Q}}
\def\C{\mathbb{C}}
\def\R{\mathbb{R}}
\def\N{\mathbb{N}}
\def\P{\mathbb{P}}
\def\T{\mathbb{T}}
\def\K{\mathbb{K}}
\def\Dc{\mathcal{D}}
\def\modp{\, (\text{mod }p)}
\def\modN{\, (\text{mod }N)}
\def\modq{\, (\text{mod }q)}
\def\modone{\, (\text{mod }1)}
\def\Zn{\mathbb{Z}/N \mathbb{Z}}
\def\Zp{\mathbb{Z}/p \mathbb{Z}}
\def\Zan{a^{-n}\mathbb{Z}/ \mathbb{Z}}
\def\Zal{a^{-l} \Z / \Z}
\def\Pr{\text{Pr}}
\def\leftsize{\left| \left\{}
\def\rightsize{\right\} \right|}

\begin{abstract}
We present an approach over arbitrary fields to bound the degree of intersection of families of varieties in terms of how these concentrate on algebraic sets of smaller codimension. This provides in particular a substantial extension of the method of degree-reduction that enables it to deal efficiently with higher-dimensional problems and also with high-degree varieties. We obtain sharp bounds that are new even in the case of lines in $\R^n$ and show that besides doubly-ruled varieties, only a certain rare family of varieties can be relevant for the study of incidence questions.
\end{abstract}

\maketitle

\tableofcontents

\section{Introduction}

\subsection{Statement of results}

The fundamental result of Szemerédi and Trotter \cite{ST} gives an optimal bound for the number of incidences that can occur between a set of points $S$ and a set of lines $T$ in $\R^2$ in terms of the cardinality of these sets. It is clear that one cannot obtain a better estimate over $\R^3$, since placing all the points and lines inside some hyperplane one can reconstruct any configuration that occurs in $\R^2$. Nevertheless, it was shown in a landmark paper \cite{GK} of Guth and Katz that this is essentially the only obstruction in that a much stronger bound can be attained as long as no hypersurface of degree at most $2$ contains more than $O(|T|^{1/2})$ elements from $T$. This naturally raises the following question: 

\begin{que}
\label{A}
To what extent is the number of incidences produced by a family of varieties $T$ determined by how this family concentrates on algebraic sets of smaller co-dimension?
\end{que}

There has been a significant amount of effort expended in extending the methods of Guth and Katz and related ideas to understand this question. The picture has become better understood in $\K^3$, with $\K$ an arbitrary field \cite{EH,GK,GZ2,K}, with progress also made in $\R^4$ \cite{GZ,SS}, but much less has been achieved for general choices of $\K^n$. While there are interesting results that have been established in this setting, they tend to require the rather strong assumption that the family of varieties $T$ being studied only has $O(1)$ elements lying in a low degree variety of smaller co-dimension \cite{DS,SSS} or morally equivalent conditions of transversality \cite{KSS,Q,SoT,Z2}. For comparison, if for instance $T$ is a set of lines in $\K^n$, one should expect to be able to place $\sim |T|^{\frac{m-1}{n-1}}$ elements of $T$ in an $m$-dimensional plane without paying any price in the bound.

Given a set of $l$-dimensional varieties $T$ and a variety $W$, we shall write $T_W$ for those elements of $T$ lying inside of $W$. Perhaps the most natural way to measure the concentration of a general family of varieties $T$ is to consider the quantities
 $$ \mathcal{D}_m(T) = \max \left\{ \frac{\deg(T_W)}{\deg(W)} : \dim(W) = m \right\},$$
 where the maximum goes along all $m$-dimensional varieties. Notice that a typical family of $l$-dimensional varieties in $\K^n$ satisfies $\Dc_m(T) \sim \deg(T)^{\frac{m-l}{n-l}}$ for every $l \le m \le n$.
  
  In the present article we will introduce a method that produces incidence estimates with a best-possible dependence on the quantities $\Dc_m(T)$. We also show how to use additional information on the elements of $T$ to restrict the set of varieties $W$ needed in the definition of $\Dc_m(T)$ while producing the same bounds. These results allow, for example, the optimal number of lines to lie inside an $m$-dimensional plane without affecting the bounds. In particular, they provide a quantitatively strong answer to Question \ref{A} by showing that a large number of incidences can always be attributed to $T$ having a significantly high concentration on a variety of smaller codimension. 
    
  Our approach provides an alternative way of producing partitions via polynomials that does not require topological considerations and therefore works over arbitrary fields. Combining it with the usual partitioning techniques, we will also provide a correspondingly stronger bound over $\R^n$. Our partitioning method also has the unusual feature of working equally well if we are studying the incidences of a set of varieties with varieties of smaller dimension that are not necessarily points. This phenomenon is actually connected to the obstructions in extending point-incidence estimates to arbitrary dimensions, as was already noted in \cite{GZ}, although this connection is less explicit in the present work.
  
  In this article we will restrict attention to incidences between sets of varieties $S$ and $T$ with $\dim(S)=\dim(T)-1$. While we expect the approach to work in arbitrary co-dimension, this assumption simplifies considerably not only the methods but even the correct statements of the results. In larger co-dimension, besides the need of some standard assumptions to be placed in order to make non-trivial results possible, an optimal result must furthermore keep track not only of the quantities $\Dc_m(T)$ but also of some generalisations of these parameters that measure concentration in the form of degrees of intersection of specified dimension (see \cite{DS} for a similar discussion).
     
  Finally, let us remark that unlike most work in incidence geometry, we obtain results that provide an optimal dependence on $\deg(T)$ without placing a uniform bound on the degrees of the individual elements of $T$. Arguably, this makes it natural to view them as a form of quantitative intersection theory. To discuss this, suppose we are trying to estimate the number $\mathcal{P}_2(T)$ of intersection points of a set of irreducible curves $T$. Bezout's theorem gives us a bound of the form $\deg(T)^2$. On the other hand, this theorem essentially gives an equality if we can place all elements of $T$ inside of a plane and more generally, one would intuitively expect that the elements of $T$ can only have an abnormally large degree of intersection with each other if they are highly trapped inside a subvariety. We can therefore reiterate the same question as before in this context: 
  
  \begin{que}
  To what extent can we improve upon Bezout's inequality by knowing how the elements of $T$ concentrate inside varieties of smaller co-dimension?
  \end{que}
  
 Given families of varieties $S$ and $T$ with $\dim(S) \le \dim(T)$, we define the degree of incidence between $S$ and $T$ as the quantity
$$ \mathcal{I}(S,T) = \sum_{s \in S} \deg(s) | \left\{ t \in T : s \subseteq t \right\}|.$$
This is obviously the usual number of incidences when $S$ is a set of points. We have the following optimal estimate in terms of the parameters $\Dc_m(T)$, which improves on the best known bounds even in the simplest case of lines.

\begin{teo}
\label{I0}
Let $S$ and $T$ be sets of irreducible varieties in $\K^n$ with $\dim(T) =\dim(S)+1$. Then
$$ \mathcal{I}(S,T) \lesssim_n \sum_{1 \le m \le n-\dim(S)} \deg(S)^{1/m} \deg(T)^{1-1/m} \Dc_{m+\dim(S)}(T)^{1/m}.$$
\end{teo}

We say $S$ is $k$-free with respect to $T$ if any subset of $S$ of size $k$ lies in at most one element from $T$. Notice that for arbitrary sets of varieties $S$ and $T$, with $\dim(T)=\dim(S)+1$, we can always take $k=(\max_{t \in T} \deg(t))^2+1$ by Bezout's theorem. When $\K=\R$ the polynomial partitioning techniques of Guth and Katz (and the refinements given in \cite{W5}) can be used to exploit the topology of $\R^n$ when studying point-incidences, strengthening the above incidence estimate if we know $k$ to be small.

\begin{teo}
\label{R}
Let $T$ be a family of irreducible algebraic curves in $\R^n$. Then, for any set of points $S \subseteq \R^n$ that is $k$-free with respect to $T$, we have the bound
$$ \mathcal{I}(S,T) \lesssim_{n} \sum_{0 \le m \le n} k^{1-\alpha(k,m)} |S|^{\alpha(k,m)} \deg(T)^{1-\alpha(k,m)} \Dc_m (T)^{\frac{k-1}{k} \alpha(k,m)},$$
with $\alpha(k,m) = \frac{k}{m(k-1)+1}$ for $m \ge 1$ and $\alpha(k,0)=0$.
\end{teo} 

Let us write $\mathcal{P}_r(T)$ for the set of irreducible varieties of dimension $\dim(T)-1$ that lie on at least $r$ elements of $T$. The results above imply a corresponding bound for $\mathcal{P}_r(T)$ when $r > O_n(1)$. We have essentially the same estimate in the remaining range.

  \begin{teo}
\label{I1}
Let $T$ be a family of irreducible varieties in $\K^n$ of dimension $d$. Then, for every $r \ge 2$, every $k$-free subset $S \subseteq \mathcal{P}_r(T)$ satisfies
$$ \deg(S) \lesssim_n \frac{\deg(T)}{r}  \left( k^{1/2} + \sum_{1 \le m \le n-d}\left( \frac{\Dc_{m+d}(T)}{r} \right)^{1/m} \right).$$
\end{teo}

It is important to remark that the ideas of this article and most of those cited above actually have their origin in Dvir's solution of the Kakeya problem over finite fields \cite{D}, where he introduced the polynomial method in this context (see \cite{G,T0,W}). Dvir's result can be deduced from Theorem \ref{I0}, which allows us to also view it as part of this broader phenomenon of concentration estimates.

Our results are consistent with conjectures due to several authors \cite{G4,SSS,Z}. Essentially, they predict that under the additional assumption that all elements of $T$ belong to some specific families, then one can impose some corresponding restrictions on the kind of varieties $W$ used in the definition of $\Dc_m(T)$ and obtain the same results. A general result in this direction is discussed in the next part of this introduction. On the contrary, the results we have stated thus far do not place any assumptions on the elements of $T$ and constitute the best-possible estimates in this setting, as can be seen via constructions similar to those of \cite[\S 8.3]{W5}.

It may be important to remark that even in the simplest case of lines in $\R^n$ this goes beyond what was achievable by previous methods. For example, suppose $L$ is a set of lines with concentration values $\Dc_m(L)$ coinciding with those that would arise if we chose $|L|^{1-\frac{m-1}{n-1}-c}$ $m$-planes generically, with each of them containing $|L|^{\frac{m-1}{n-1}+c}$ lines that are not abnormally concentrated in subvarieties of these $m$-planes. In this case, both the current and previous approaches provide the expected bound $\mathcal{P}_2(L) \lesssim |L|^{\frac{n}{n-1}}$ if $c \le 0$, which is when the values $\Dc_m(L)$ coincide with those of a generic set of lines. On the other hand, if $c > 0$, previous methods, which are based on standard degree-reduction, are only capable of achieving a bound of the form $\mathcal{P}_2(L) \lesssim_n |L|^{\frac{n}{n-1}+c \frac{n-2}{n-m}}$ while Theorem \ref{I1} returns the expected bound $\mathcal{P}_2(L) \lesssim_n |L|^{\frac{n}{n-1}+\frac{c}{m-1}}$. Notice that the bounds only coincide when $m=2$, while the latter is substantially stronger as soon as $m \ge 3$. The comparative advantage of Theorem \ref{I1} increases further if we allow $T$ to consist of elements of large degree.

The main difference is that degree-reduction is only able to efficiently construct a hypersurface containing the given set $T$, while the present method allows one to find a variety of the appropriate dimension where the set $T$ concentrates. This is particularly important if one wishes to carry a finer analysis of the relevant varieties, as done in Theorem \ref{p} below. In fact, the main contribution of this article can be seen as a substantial extension of the method of degree-reduction to make it able to deal with higher-dimensional problems and also with high-degree varieties.

Before ending this part of the introduction, it may be worthwhile to point out that the above estimates make rigorous the idea that a family of subvarieties $T$ of an ambient variety $V$ should find it harder to intersect as the degree of $V$ gets larger, as was already noted in \cite{W5}.

\subsection{Restricted families}

It is natural to ask the following further question.

\begin{que}
If all the elements of $T$ lie on a particular family of varieties, can we obtain the same results with the definition of $\Dc_m(T)$ now restricted to varieties $W$ of a corresponding special type?
\end{que}

The answers known in the literature to this question proceed by showing that if many elements from a fixed type of varieties are incident to a point of a variety $W$ in which they are contained, this usually forces this point to have some special property with respect to $W$ that can be verified using polynomials of controlled degree. This information can in turn be used to restrict the set of varieties $W$ relevant for the study of $T$. The methods of this paper show that a similar phenomenon holds in general.

To formalise the idea just described we need to introduce some notation and definitions. Some further discussion is provided after stating Theorem \ref{p} below. Given an irreducible variety $W \subseteq \K^n$, we write $\delta(W)$ for smallest $\delta$ for which there exist polynomials $g_1,\ldots,g_r$ of degree at most $\delta$ such that $W$ is an irreducible component of their zero set $Z(g_1,\ldots,g_r)$. Equivalently, $\delta(V)$ is the smallest degree needed to set-theoretically define $V$ over a Zariski-dense subset. We say $W$ is entangled to another irreducible variety $W'$ if a polynomial of degree $\lesssim_n \delta(W) + \delta(W')$ vanishes identically on $W$ if and only if it does on $W'$. It is possible to see that a variety $W \subseteq \K^n$ can only be entangled to $O_n(1)$ other irreducible varieties $W'$ and that they all satisfy $\delta(W) \sim_n \delta(W')$, $\deg(W) \sim_n \deg(W')$ and $\dim(W)=\dim(W')$.

 As previously remarked, if a point of $W$ is incident to many subvarieties of $W$ of a specific kind, this usually forces this point to have some special property with respect to $W$. Because of this, given some property $p$ an $(l-1)$-dimensional irreducible variety may have with respect to a given variety (e.g. being a singular point), we say that a family of $l$-dimensional varieties $T \subseteq \K^n$ is a $p$-set if for every irreducible variety $W$ with $\dim(T) < \dim(W) < n$, there is some $r=O_n(1)$ such that every element of $\mathcal{P}_r(T_W)$ has property $p$ with respect to $W$. 
 
 Properties relevant in incidence geometry involve the behaviour on tangent spaces and this makes them verifiable using polynomials built out of any set of polynomials that can be used to define this tangent space. Since for a generic choice of $W \subseteq \K^n$ we can find polynomials of degree at most $\delta(W)$ defining the tangent space on a Zariski-dense subset of $W$, most natural properties for incidences will be verifiable using polynomials of degree $O_n(\delta(W))$. We therefore say $p$ is a $\delta$-property of an irreducible variety $W$ if the set of $(l-1)$-dimensional irreducible subvarieties of $W$ having property $p$ with respect to $W$ lie in the zero set of a polynomial of degree $O_n(\delta(W))$ that does not vanish identically on $W$. Varieties for which $p$ is not a $\delta$-property will play a special role. We write $\mathcal{C}_p$ for those irreducible varieties that are entangled to an irreducible variety for which $p$ is not a $\delta$-property.

Notice that the usual properties that are relevant in incidence geometry, like a point of a variety being singular, flecnodal, double-flecnodal, flexy, etc., can be seen to be $\delta$-properties of hypersurfaces in $\K^3$, except precisely for those hypersurfaces that turn out to be relevant for the study of the corresponding incidence problems (e.g. doubly ruled surfaces). With this in mind, let us write $\Dc_m^p(T)$ for the expression $\Dc_m(T)$, but where the maximum is now restricted to varieties $W$ in $\mathcal{C}_p$.

We have the following result, showing that in the general situation just described, we can indeed guarantee that it is only the rather restricted set of varieties $\mathcal{C}_p$ that can be relevant for the study of the incidences of $T$.

\begin{teo}
\label{p}
In both Theorem \ref{I0} and Theorem \ref{R}, if $T$ is a $p$-set, then we can replace $\Dc_m(T)$ by $\Dc_m^p(T)$.
\end{teo}

The intention of this formulation is mostly to be indicative of what kind of results are easy to attain using our methods. It is possible that further progress can be achieved by mixing our arguments with certain tools with which they are compatible, like truncated partitionings. Nevertheless, once Theorem \ref{p} is combined with corresponding concentration estimates or classification of special surfaces, it already allows one to unify a large number of results in the literature, including the Szemerédi-Trotter theorem \cite{ST}, the Pach-Sharir theorem \cite{PS}, Dvir's theorem \cite{D}, the joints problem \cite{GK0}, the Guth-Katz theorem \cite{GK}, Köllar's bound \cite{K}, Bourgain's conjecture over arbitrary fields as established by Ellenberg and Hablicsek \cite{EH}, the bounds of Dvir and Gopi \cite{DG} and Hablicsek and Scherr \cite{HS}, as well as the result of Guth and Zahl on constructible families \cite{GZ2}. As it was mentioned before, it should be possible to extend the results to the whole range $\dim(S) \le \dim(T)$, thus also recovering a result of the author \cite{W5} subsuming further previous work \cite{BK,BS,CEGSW,ES,FPSSZ,KMSS,LSZ,Z1}.

We finish this part of the introduction with some last comments on Theorem \ref{p}. Let us write $\mathcal{J}_{\C^n}$ for the set of irreducible varieties $W \subseteq \C^n$ having no choice of $x \in W$ and $g_1,\ldots,g_{n-\dim(W)} \in I(W)$ of degree $\lesssim_n \delta(W)$ such that the corresponding Jacobian matrix $\text{Jac}(g_1,\ldots,g_{n-\dim(W)})(x)$ has maximal rank. Of course, the elements of $\mathcal{J}_{\C^n}$ are rather rare if they exist. On the other hand, for a property $p$, write $\mathcal{C}_p^{\ast}$ for those irreducible varieties $W$ for which a generic $(l-1)$-dimensional irreducible subvariety of $W$ has property $p$ with respect to $W$. Notice that, given a set of varieties $T$ in some class, it is not hard to show that one can take $p$ to be a (doubly-)flecnodal property with respect to which $T$ is a $p$-set and which forces the elements of $\mathcal{C}_p^{\ast}$ to be generically doubly ruled by elements in the class of $T$ (similarly as it is done in \cite{GZ2}). Therefore, the ideal classification one could hope for a given class of incidence problems in $\R^n$ (e.g. incidences of points and lines) is that we only care about the concentration of $T$ on elements of $\mathcal{C}_p^{\ast}$ (e.g. \cite[Conjecture 4.1]{G4}). What Theorem \ref{p} essentially shows is that we only care about the concentration of $T$ on elements of $\mathcal{C}_p^{\ast} \cup \mathcal{J}_{\C^n}$. In particular, for any class of problems the exceptions always lie on the fixed family $\mathcal{J}_{\C^n}$. 

Theorem \ref{p} also shows that the larger the degree of a variety $W \in \mathcal{C}_p^{\ast} \cup \mathcal{J}_{\C^n}$ is, the larger $\deg(T_W)$ needs to be for $W$ to be significant in the incidence problem. It would of course be interesting to know examples of varieties $W$ in $\mathcal{J}_{\C^n}$ but unfortunately the author knows none. The following is a very simple question in this direction: does there exist an absolute constant $C$ depending only on $n$ such that every non-ruled irreducible two-dimensional variety $V$ in $\C^n$ contains at most $C \delta(V) \deg(V)$ lines?

\subsection{Overview of the method}
\label{OM}

We now provide an overview of the method. For concreteness, let us first suppose we are trying to estimate the number of $2$-rich points produced by a set of lines $T$ in $\K^3$ in the form of Theorem \ref{I1}. For a generic choice of $T$, these lines should not exhibit any abnormal concentration inside of a surface of $\K^3$ and therefore we expect a bound of the form $\deg(P_2(T)) \le K |T|^{3/2}$ with $K=O(1)$ in this case.

Let us assume for the moment that every set of lines behaves similarly to a generic choice of lines and show a way to establish this estimate if that was the case. If $|T|=O(1)$ the bound is obvious, so let us proceed by induction on the size of $T$. To do this, we would like to find a polynomial $f$ that allows us to partition $T$ in an appropriate way. Since $T$ is generic, this is easy. Indeed, if we pick a subset $T_1 \subseteq T$ of size $\tau |T|$ for some small $\tau > 0$, we can find a polynomial $f$ of degree $\lesssim \tau^{1/2} |T|^{1/2}$ vanishing on $T_1$. If $\tau$ is sufficiently small, we see that this polynomial cannot vanish on all of $T$ as this would be an abnormally small polynomial vanishing on $T$ and thus contradict our assumption that $T$ behaves like a generic choice of lines. In fact, in a generic situation we can guarantee that the elements of $T_1$ are the only ones of $T$ inside of $Z(f)$, so let us assume this for the sake of discussion. Let us now write $T_2=T \setminus T_1$. We see that every element of $\mathcal{P}_2(T) \setminus \left( \mathcal{P}_2(T_1) \cup \mathcal{P}_2(T_2) \right)$ must lie inside of an element of $T_1$, therefore inside of $Z(f)$, and also inside of an element of $T_2$. But an element of $T_2$ can only intersect $Z(f)$ in $\deg(f) \lesssim \tau^{1/2} |T|^{1/2}$ points. We therefore conclude by induction that
\begin{equation}
\begin{aligned}
 \mathcal{P}_2(T) &\le \mathcal{P}_2(T_1) + \mathcal{P}_2(T_2) + O(\tau^{1/2} |T_2| |T|^{1/2}) \\
 &\le K |T_1|^{3/2} + K |T_2|^{3/2} + O_{\tau}(|T|^{3/2}) \\
 &\le |T|^{3/2} (K \tau^{3/2} + K (1-\tau)^{3/2} + O_{\tau}(1)) \\
 &\le K |T|^{3/2},
 \end{aligned}
 \end{equation}
 as long as $\tau \sim 1$ and $K$ is chosen sufficiently large with respect to $\tau$.
 
 This approach is appealing because it suggests a way of using a form of polynomial partitionings that does not require us to exploit the topology of $\R$ and therefore is viable over arbitrary fields. Furthermore, this approach also overcomes some obstructions that arise when trying to extend the real partitioning techniques to arbitrary dimensions. Of course, given that the heuristic we have given uses properties of generic sets that do not hold in general, it is not clear at all that something like this strategy can be carried out in practice.
 
 Given a finite set of varieties $T$ and a polynomial $f$, we shall write $T_f$ for those elements of $T$ contained inside of $Z(f)$. We will show that the approach we have described can be turned into a rigorous scheme, allowing us to control the degree of intersection between families of varieties by a systematic use of the following very simple observation.

  \begin{lema}
 \label{CII}
 Let $\K$ be an algebraically closed field. Let $f \in \K[x_1,\ldots,x_n]$ and let $S$, $T$ be finite sets of irreducible varieties in $\K^n$ with $\dim(T)=\dim(S)+1$. Then
 \begin{equation}
 \label{Ceq}
  \mathcal{I}(S,T) \le \mathcal{I}(S_f,T_f) + \mathcal{I}(S \setminus S_f, T \setminus T_f) + \deg(T \setminus T_f) \deg(f).
  \end{equation}
 \end{lema}
 
 \begin{proof}
 Since we clearly have $\mathcal{I}(S \setminus S_f, T_f) = 0$, it suffices to show that $\mathcal{I}(S_f,T \setminus T_f) \le \deg(T \setminus T_f) \deg(f)$, which is an obvious consequence of Bezout's theorem (Lemma \ref{Bezout} below).
 \end{proof}
 
 That not every family of algebraic varieties is generic is already accounted for in the statement of Theorem \ref{I0}, which gives an optimal dependence on how the family concentrates on algebraic sets of smaller co-dimension. However, we cannot use Lemma \ref{CII} to establish Theorem \ref{I0} without some enhanced control on the genericity of $T_1$ or $T_2$ via the parameters $\Dc_m(T_i)$. Otherwise, evaluating the first two terms of (\ref{Ceq}) by induction and applying Hölder's inequality, we already obtain the bound for $\mathcal{I}(S,T)$ we wish to attain while we still have not dealt with the third term.
 
 To overcome this, we will use an adequate iteration of Lemma \ref{CII} to construct a finite partition of $T$ such that, for some component $T_r \subseteq T$ of this partition, we can obtain a stronger incidence estimate that allows us to absorb all the additional terms that arise from the use of Lemma \ref{CII}. Starting with a family of varieties $T$ that is highly concentrated inside a certain irreducible variety $V$ of dimension $d$, we accomplish the above task by showing that we can always use Lemma \ref{CII} to partition $T$ in a controlled manner such that an element $T_r$ of this partition has size comparable to $T$ and satisfies the following dichotomy: either $T_r$ is less-concentrated than $T$ in the sense that $\Dc_m(T_r) < \Dc_m(T)/2$ for an adequate value of $m \ge d$, or in fact $T_r$ is highly concentrated on an irreducible variety $W \subseteq V$ of dimension $d-1$. This is established by exploiting the Siegel-type lemmas proved by the author in \cite{W5}.
 
 Even then, this will only work if at each application of Lemma \ref{CII} we are using a partitioning polynomial $f$ of adequately small degree. To this end, our methods crucially ensure that in the setting discussed in the previous paragraph, if $D$ is the smallest degree of a polynomial vanishing on all elements of $T$ without vanishing identically on $V$, then the partitioning polynomials $f$ can always be taken to have degree $\lesssim_n D$. 
 
This preservation of the degree makes the method particularly flexible, since it allows us to incorporate any a priori bound we might have on the degree of a subvariety containing $T$. In particular, since the usual partitioning techniques over $\R^n$ (and the generalisations to arbitrary varieties $V \subseteq \R^n$ obtained in \cite{W5}) allow us to estimate all incidences occurring outside of a subvariety of controlled degree, this information can simply be used as an input in the above method to produce stronger bounds. This leads to Theorem \ref{R}. Similarly, the subvariety arising from the implicit polynomial in the definition of a $p$-set can be used to obtain Theorem \ref{p}. The approach would be equally well-suited to any other way we have of bounding the relative degree, including truncated partitionings with low or medium degree polynomials.

\subsection{Organisation of the paper}

The rest of this article is organised as follows. In Section \ref{2} we review some algebraic preliminaries as well as certain results of \cite{W5} that we shall need. In Section \ref{3} we introduce and discuss the main definitions used in the proof and show how, given a set of varieties $T$, we can use a polynomial of small degree to pass to a subset of $T$ with better concentration properties. In Section \ref{4} we develop the main tools needed to handle the partitioning procedure we shall use. Finally, the proof of Theorem \ref{I0} is carried out in Section \ref{5}, while in Section \ref{6} we show how to modify the argument to establish Theorem \ref{R} and Theorem \ref{p}.

\begin{ack}
Part of this work was carried while the author was a Clay Research Fellow and a Fellow of Merton College at the University of Oxford.
\end{ack}

\section{Preliminaries}
\label{2}

In this section we review some of the notation (\S \ref{Not}) and algebraic preliminaries (\S \ref{AP}) we shall need. In \S \ref{ERD} we state some results from \cite{W5} that will be used in the rest of the article. Finally, \S \ref{RBP} contains a summary of the relationship between some of the parameters used in the proof of Theorem \ref{I0}.

\subsection{Notation}
\label{Not}

Given parameters $a_1,\ldots,a_r$ we shall use the asymptotic notations $X \lesssim_{a_1,\ldots,a_r} Y$ or $X=O_{a_1,\ldots,a_r}(Y)$ to mean that there exists some constant $C$ depending only on $a_1,\ldots,a_r$ such that $X \le C Y$. We write $X \sim_{a_1,\ldots,a_r} Y$ if $X \lesssim_{a_1,\ldots,a_r} Y \lesssim_{a_1,\ldots,a_r} X$. If $A=\left\{ a_1, \ldots, a_r \right\}$ is a set of parameters, we will abbreviate $X \lesssim_{a_1,\ldots,a_r} Y$ as $X \lesssim_A Y$. We shall write $|S|$ for the cardinality of a set $S$.

Given a field $\K$ and polynomials $f_1,\ldots,f_r \in \K[x_1,\ldots,x_n]$ we will write
$$ Z(f_1,\ldots,f_r) = \left\{ x \in \K^n : f_1(x) = \cdots = f_r(x) = 0 \right\},$$
for the corresponding zero set. For an algebraic closed field $\K$ and an irreducible algebraic variety $V \subseteq \K^n$ we write $\deg(V)$ for the degree of its projective closure and more generally, for an algebraic set $V$ with irreducible components $V_1,\ldots,V_s$ we write $\deg(V) = \sum_{i=1}^s \deg(V_i)$. By an algebraic set of dimension $d$ we mean an algebraic set all of whose irreducible components have dimension $d$. In particular, if we write $\dim(T)=d$ we mean that every irreducible component of $T$ has dimension $d$.

 Given a set of varieties $T$ and a polynomial $f$, we write $T_{f}$ for those elements of $T$ contained inside of $Z(f)$. Similarly, if $W$ is an algebraic set, we write $T_W$ for those elements of $T$ that lie inside of $W$.

\subsection{Algebraic preliminaries}
\label{AP}

We will be using the following form of Bezout's inequality  \cite[Theorem 7.7]{Hart}.

\begin{lema}[Bezout's inequality]
\label{Bezout}
Let $\K$ be an algebraically closed field. Let $W \subseteq \K^n$ be an irreducible variety and $f_1,\ldots,f_s \in \K[x_1,\ldots,x_n]$ polynomials. Write $Z_1,\ldots,Z_r$ for the irreducible components of $Z(f_1,\ldots,f_s) \cap W$. Then
$$ \sum_{i=1}^r \deg(Z_i) \le \deg(W) \prod_{j=1}^s \deg(f_j).$$
\end{lema}

As in \cite{W5}, given an irreducible variety $V \subseteq \K^n$ over an algebraically closed field $\K$, we will need to consider the following quantities.

\begin{defi}[Partial degree]
\label{delta}
For an irreducible algebraic variety $V \subseteq \K^n$ and every $1 \le i \le n-\dim(V)$ we let $\delta_i(V)$ stand for the minimal integer $\delta$ for which we can find a finite set of polynomials $g_1,\ldots,g_t \in \K[x_1,\ldots,x_n]$ of degree at most $\delta$ such that $V \subseteq Z(g_1,\ldots,g_t)$ and the highest dimension of an irreducible component of $Z(g_1,\ldots,g_t)$ containing $V$ is equal to $n-i$. We sometimes abbreviate $\delta_{n-\dim(V)}(V)$ as $\delta(V)$ and call this the partial degree of $V$. By convention we also write $\delta_0(V)=1$ and $\delta_i(V)=\infty$ for every $i>n-\dim(V)$.
\end{defi} 

It is immediate to verify that these quantities satisfy the following simple relation.

\begin{lema}
\label{dorder}
For every variety $V$ we have $\delta_i(V) \ge \delta_{i-1}(V)$ for every $i$.
\end{lema}

If $V$ is not irreducible we will use the following variant of the above definition.

\begin{defi}
For an algebraic set $V \subseteq \K^n$ having all its irreducible components of the same dimension we write $\delta(V)$ for the smallest integer $\delta$ for which we can find polynomials $g_1,\ldots,g_t$ of degree at most $\delta$ such that every irreducible component of $V$ is also an irreducible component of $Z(g_1,\ldots,g_t)$.
\end{defi} 

Given an irreducible algebraic variety $V \subseteq \K^n$ of dimension $d$ and any $1 \le i \le n-d$, write 
$$\Delta_i(V) = \max \left\{ \frac{\deg(V)}{\delta_{i+1}(V) \cdots \delta_{n-d}(V)} , 1 \right\}.$$
Notice that we have
\begin{equation}
\label{basic1}
\Delta_{i+1}(V) \le \delta_{i+1}(V)\Delta_{i}(V),
\end{equation}
with equality holding whenever $\Delta_{i}(V)>1$. In fact, we have the following estimate.

\begin{lema}
\label{Dta}
For any irreducible variety $V \subseteq \K^n$ of dimension $d$ and any $1 \le i \le n-d$, it is
$$ \Delta_{i}(V) \sim_n \prod_{j=1}^{i} \delta_j(V).$$
\end{lema}

\begin{proof}
This is a consequence of \cite[Theorem 5.5]{W5}.
\end{proof}

\subsection{Estimates on relative degrees}
\label{ERD}

We will now state some results from \cite{W5} that will be needed during our proofs. We refer to that article for further discussion of these estimates. 

For an irreducible algebraic variety $V \subseteq \K^n$, we say a non-negative integer $i$ is admissible with respect to $V$ if $\delta_{i+1}(V) > 2 i \delta_i(V)$. We will consider intervals of the form
$$ \mathcal{R}_{s,\tau}^l(V) = [\tau \delta_s(V)^{n-(s+l)} \Delta_{s}(V), \tau \delta_{s+1}(V)^{n-(s+l)} \Delta_s (V) ],$$
with $\tau > 0$ a real number, integers $0 \le l < n-s$ and $V \subseteq \K^n$ an irreducible algebraic variety. 

The following observation follows immediately from the definition of $\Delta_i$ and the fact that given a positive integer $s$, if $t$ is the smallest admissible integer with $s \le t$, then $\delta_s(V) \gtrsim_n \delta_t(V)$.

\begin{lema}
\label{coverN}
Let $V \subseteq \K^n$ be an irreducible algebraic variety of dimension $d$. For any integer $l<d$ and $0 < \varepsilon < 1$, we can find $\varepsilon \lesssim_n \tau_1,\ldots,\tau_{n-d} \le \varepsilon$ such that $\R_{\ge 0}$ is covered by the sets $\mathcal{R}_{s,\tau_s}^l(V)$ with $s$ admissible.
 \end{lema}
 
 Given algebraic sets $T$ and $V$ of dimension $l$ and $d$ respectively, the following result \cite[Theorem 4.6]{W5} gives an optimal bound for the degree of a polynomial vanishing on $T$ without vanishing identically on $V$. Its proof is based on estimates on ideals originating in \cite{CP}.

\begin{teo}
\label{GSiegel0}
Let $\K$ be an algebraically closed field. Let $0 \le l < d \le n$ be integers and $\tau_l >0$ a sufficiently small constant with respect to $n$. Let $T$ be a finite set of $l$-dimensional irreducible algebraic varieties in $\K^n$ and $V$ a $d$-dimensional irreducible algebraic variety in $\K^n$. Let $0 \le s \le n-d$ be an admissible integer with respect to $V$ with $\deg(T) \in \mathcal{R}_{s,\tau_l}^l(V)$. Then, there exists some polynomial $P \in \K[x_1,\ldots,x_n]$ of degree at most
\begin{equation}
\label{Rbound}
 \lesssim_{n,\tau_l} \left( \frac{\deg(T)}{\Delta_{s}(V)} \right)^{\frac{1}{n-(s+l)}},
 \end{equation}
vanishing at all elements of $T$ without vanishing identically on $V$.
\end{teo}

We will use the following simple observation.

\begin{lema}
\label{simplelinear}
Let $\K$ be an algebraically closed field. Let $V_1,\ldots,V_r$ be subsets of $\K^n$ and let $S \subseteq \bigcap_{i=1}^r V_i$. Let $f_1,\ldots,f_r$ be polynomials such that they all vanish on $S$ but each $f_i$ does not vanish identicaly on $V_i$. Then, there is some $\K$-linear combination $f=c_1 f_1 + \ldots + c_r f_r$ such that $f$ vanishes on $S$ but does not vanish identically on any $V_i$.
\end{lema}

The results we have stated so far imply the following estimate, which is phrased in a way that is particularly suitable to be applied throughout the paper.

\begin{lema}
\label{BL}
Let $\K$ be an algebraically closed field. Let $0 \le l < m \le n$ be integers and let $0 < \epsilon < 1$ and $A \ge 1$ be some given parameters. Let $r=O_n(1)$ and let $W=W_1 \cup \ldots \cup W_r$ be an algebraic set in $\K^n$ with each $W_i$ an irreducible variety of dimension $m$ with $\delta_q(W_i) \sim_n \delta_q(W_j)$ for every choice of $1 \le i \le j \le r$, $1 \le q \le n-m$. Let $X$ be an algebraic set of dimension $l$ with $\deg(X) \le \epsilon \deg(W) (A\delta(W))^{m-l}$. Then we can find a polynomial of degree $\le C \epsilon^{1/(n-l)} A \delta(W)$ that vanishes on $X$ without vanishing on any $W_i$, $1 \le i \le r$, where $C \lesssim_n 1$ is an absolute constant independent of $\epsilon$ and $A$.
\end{lema}

\begin{proof}
Since $\delta_q(W_i) \sim_n \delta_q(W_j)$ for every $1 \le i \le j \le r$ and $1 \le q \le n-m$, we know by Theorem \ref{GSiegel0} (and Lemma \ref{coverN}) that there exists some $m \le s \le n$ such that for every $W_i$, $1 \le i \le r$, we can find a polynomial $f_i$ vanishing on $X$ without vanishing on $W_i$ satisfying
$$ \deg(f_i) \lesssim_n \left( \frac{\deg(X)}{\Delta_{n-s}(W_i)} \right)^{1/(s-l)} .$$
The result then follows from Lemma \ref{simplelinear}, our bound on $\deg(X)$ and the fact that by Lemma \ref{dorder} and Lemma \ref{Dta} it is
$$\Delta_{n-s}(W_i) \delta(W)^{s-m} \gtrsim_n \deg(W).$$
\end{proof}

The estimate above gives a good bound for the smallest degree of a polynomial vanishing on $X$ without vanishing on a given algebraic set $W$. The precise value of the smallest such degree can of course be smaller than this upper bound in certain situations and will play an important role throughout the paper. This motivates the following definition.

\begin{defi}[Relative degree]
\label{RD}
Given algebraic sets $X,W$ of dimension $l$ and $m$ respectively, with $l<m$, we define the relative degree $\deg_R(X,W)$ of $X$ with respect to $W$ to be the smallest degree of a polynomial vanishing on $X$ without vanishing on any irreducible component of $W$.
\end{defi} 

So, for example, if $X$ and $W$ are algebraic sets satisfying the hypothesis of the statement of Lemma \ref{BL}, we have
$$ \deg_R(X,W) \lesssim_n \epsilon^{1/(n-l)} A \delta(W).$$

Let us finish this subsection recalling two other results from \cite{W5} that will be needed later. Given an irreducible algebraic variety $V \subseteq \C^n$ and an integer $M$, we write $i_V(M)$ for the smallest admissible $i$ such that $M^{n-i} \Delta_i(V) \in \mathcal{R}_{i,c}^{0}(V)$, where $c \gtrsim_n 1$ is a sufficiently small constant.
Clearly, we have
\begin{equation}
\label{ivmrange}
\delta_{i_V(M)}(V) \lesssim_{n} M \lesssim_n \delta_{i_V(M)+1}(V).
\end{equation}

If $V$ is an irreducible complex variety, we write $V(\R)$ for its real points. We have the following result \cite[Theorem 3.1]{W5}. 

\begin{teo}[Polynomial partitioning for varieties]
\label{30par}
Let $V \subseteq \C^n$ be an irreducible algebraic variety of dimension $d$ and $S$ a finite set of points inside of $V(\R)$. Then, given any integer $M \ge 1$, we can find some polynomial $g \in \R[x_1,\ldots,x_n] \setminus I(V)$ of degree $O_{n}(M)$ such that each connected component of $\R^n \setminus Z(P)$ contains 
$$\lesssim_{n} \frac{|S|}{M^{n-i_V(M)}\Delta_{i_V(M)}(V)}$$
elements of $S$.
\end{teo}

Finally, we have the following estimate \cite[Theorem 1.5]{W5} in the spirit of results of Milnor-Thom \cite{M,OP,T} and Barone-Basu \cite{BB0,BB}.

\begin{teo}
\label{T5}
Let $V \subseteq \C^n$ be an irreducible algebraic variety of dimension $d$ and $P \in \R[x_1,\ldots,x_n]$. Then $V(\R)$ intersects $\lesssim_n \deg(V) \deg(P)^d$ connected components of $\R^n \setminus Z(P)$.
\end{teo}

We remark that both Theorem \ref{30par} and Theorem \ref{T5} were originally conjectured in \cite{BS}.

\subsection{Relation between the parameters}
\label{RBP}

To facilitate the reading, we now briefly summarise the relationships between some of the parameters involved in the proof of Theorem \ref{I0}. When working over $\K^n$ with a set $T$ of varieties of dimension $l$, for every $l \le k < n$ and $0 \le i \le 8$, we will let $b_i^{(k)} \sim_n 1$ be appropriately chosen constants. We will write $B^{(n)}=\left\{ n \right\}$ and more generally
$$B^{(h)} = \left\{ b_i^{(k)} : 0 \le i \le 8 \, , \, h \le k < n \right\} \cup \left\{ n \right\},$$
for every $l \le h <n$. These parameters will be obtained recursively, with the constants $b_i^{(h)}$, $0 \le i \le 8$, depending solely on $B^{(h+1)}$, so in particular they will satisfy $b_i^{(h)} \sim_{B^{(h+1)}} 1 \sim_n 1$.  The parameter $b_0^{(h)}$ in particular will be chosen to be sufficiently small with respect to the other parameters in $B^{(h)}$. 

We will also consider parameters $a_{k+1}^{(h)}$ for every choice of $l \le h \le k < n$. They will be such that $a_{k+1}^{(h)} < a_{k'+1}^{(h')}$ if and only if $(k,h) > (k',h')$ under lexicographical order. They will be chosen to be sufficiently small with respect to each other under this ordering and furthermore, $a_{l+1}^{(l)}$ will be sufficiently small with respect to $B^{(l)}$. 

Writing $K$ for the implicit constant in Theorem \ref{I0}, we have that $K$ will be chosen sufficiently large with respect to all the parameters above.

\section{Levels of reduction}
\label{3}

This section is organised as follows. In \S \ref{31} we introduce several definitions that will play an important role during the rest of the article and discuss their motivation. In \S \ref{32} we show how given a finite family $T$ of algebraic varieties, upon discarding some elements from $T$ lying inside a polynomial of small degree we can obtain a subset with significantly better concentration properties than $T$. Finally, \S \ref{33} contains some brief observations on certain relative degrees we shall need.
 
 \subsection{Reductions and good partitions} 
 \label{31}
 
 Let us being this section formalising the kind of polynomial partitions we are interested in, following the strategy outlined in \S \ref{OM}.
 
   \begin{defi}[Good partitions]
 Let $T$ be a set of $l$-dimensional varieties in $\K^n$ and let $W$ be an $m$-dimensional algebraic set with $T \subseteq W$. We say $T$ admits a $\tau$-good partition over $W$ if there exists a polynomial $f$ of degree at most $\deg_R(T,W)$ such that 
 $$\tau \deg(T) < \deg(T_{f}) < \deg(T).$$ 
\end{defi}
 
 We now introduce another definition that will play an important role in the proof of Theorem \ref{I0}. This definition is given in terms of parameters that will be obtained recursively during the rest of the arguments and will satisfy the relationships discussed in \S \ref{RBP}. Since it is a rather technical definition, after stating it we provide an informal description of the concept we are trying to capture.
 
  \begin{defi}[Reductions]
  \label{reduction}
 Let $T$ be a set of $l$-dimensional varieties in $\K^n$. By convention, we say every such $T$ reduces to level $n$ and we write $W_1^{(n)}=\K^n$. In general, given $l \le h < n$, we say $T$ reduces to level $h$ if there exists a sequence of polynomials $P_1, \ldots, P_{n-h}$ such that for every $h \le k < n$ we have that
  $$Z(P_1,\ldots,P_{n-k}) = W_1^{(k)} \cup \ldots \cup W_{r_k}^{(k)} \cup Y^{(k)} := W^{(k)} \cup Y^{(k)},$$
 where $W_i^{(k)}$, $1 \le i \le r_k$, are irreducible varieties of dimension $k$ with $W^{(k)} := W_1^{(k)} \cup \ldots \cup W_{r_k}^{(k)}$, $T \subseteq W^{(k)}$, $W^{(k)} \subseteq W^{(k+1)}$, $T_{W_i^{(k)}} \neq \emptyset$ for every $1 \le i \le r_k$ and $Y^{(k)}$ is an algebraic set that does not contain any element from $T$. Furthermore, we ask that $W^{(k)}$ is such that every $k$-dimensional algebraic set $W$ containing $T$ satisfies 
 \begin{equation}
 \label{hmin}
 \deg(W) \ge b_0^{(h)} \deg(W^{(k)}).
 \end{equation}
 and that there is no $a_{k+1}^{(h)}$-good partition of $T$ over $W^{(k+1)}$. Finally, we require that for every choice of $1 \le i,i' \le r_k$, $1 \le j \le r_{k+1}$, $h \le k < n$, we have the estimates
 \begin{equation}
 \begin{aligned}
 \label{lines}
 r_k &\le b_1^{(h)},\\
  \deg(W_i^{(k)}) &\ge b_2^{(h)} \deg(W_j^{(k+1)}) \deg(P_{n-k}), \\
  \deg(P_{n-k}) &\ge b_3^{(h)}\deg(P_{n-k-1}), \\
  b_4^{(h)} \deg(P_{n-k}) \le \, &\delta(W_i^{(k)}) \le b_5^{(h)} \deg(P_{n-k}), \\
  \deg(W_i^{(k)}) &\le b_6^{(h)} \deg(W_{i'}^{(k)}), \\
   \deg(T) &\ge b_7^{(h)} \deg(W_i^{(k)}) \deg(P_{n-k})^{k-l}. \\
     \end{aligned}
  \end{equation}
 \end{defi}
 
 Essentially, we are saying $T$ reduces to level $h$ if it is contained inside an irreducible variety $W^{(h)}$ of dimension $h$. Since, given $l < h$, every finite family of $l$-dimensional varieties $T$ is contained in an $h$-dimensional variety $W^{(h)}$ of sufficiently high-degree, we need to ensure for this definition to be relevant that the degree of $W^{(h)}$ is small compared with $T$, which is why we impose the last condition in (\ref{lines}).
 
 The reason why instead of actually asking $W^{(h)}$ to be irreducible we need to allow it to be the union of (a bounded number of) irreducible varieties is more technical. Our methods are solely based on dimension counting arguments and degree considerations, yet there may exist varieties that are essentially undistinguishable from this point of view. This can be seen to be related to the simple fact that not every variety is a complete intersection. Nevertheless, the fact itself that they are undistinguishable guarantees that they essentially behave like a single irreducible variety and therefore they do not significantly affect our methods. In particular, there can only be a uniformly bounded number of them for each choice of one of its members. The reader might as well think of all the $W_i^{(h)}$ as a single irreducible variety $W^{(h)}$ on a first reading.
 
  We also require that there is no good partition of $T$ over $W^{(k+1)}$ for any $h \le k < n$. Morally, this means that once we are studying $T$ as a subset of $W^{(k+1)}$, the relevance of $W^{(k)}$ for the study of $T$ is unavoidable since we will have no way of carrying the argument given in \S \ref{OM} without passing to this subvariety. A related conclusion will be established rigorously in Lemma \ref{SI} below.
 
 The remaining conditions are imposed to guarantee that each $W^{(k)}$ is essentially a minimal choice among $k$-dimensional irreducible varieties (condition (\ref{hmin})) and that the polynomials $P_i$ we are using to define $W^{(h)}$ satisfy similar minimality conditions with respect to this variety.
  
 If $T$ satisfies the definition of reduction at level $h$ with slightly better parameters, then one can see that any large subset of $T$ will also reduce to level $h$. Clearly, this robustness is lost if we iterate this procedure indefinitely. However, we will only need to pass to a large subset once during the proof and this motivates the following definition.
 
 \begin{defi}[Strong reduction]
 \label{SR}
 Let $T$ be a set of $l$-dimensional varieties in $\K^n$. By convention, we say every such $T$ reduces strongly to level $n$ and we write $W_1^{(n)}=\K^n$. Given $l \le h < n$, we say $T$ reduces strongly to level $h$ if Definition \ref{reduction} holds with the stronger requirements that for every $h \le k < n$, in the notation of that definition, 
  \begin{equation}
  \label{barmin}
    \deg(W) \ge b_8^{(h)} \deg(W^{(k)}),
    \end{equation}
  for every $k$-dimensional algebraic set $W$ containing $T$, that there is no $a_{k+1}^{(h)}/2$-good partition of $T$ over $W^{(k+1)}$ and  that
  $$    \deg(T) \ge 2b_7^{(h)} \deg(W_i^{(k)}) \deg(P_{n-k})^{k-l}. $$
  \end{defi}
  
  Notice we have the following estimate, which we shall use later to apply Lemma \ref{BL}.
  
  \begin{lema}
  \label{q}
  Let the notation be as in Definition \ref{reduction}. Then $\delta_q(W_i^{(k)}) \sim_{B^{(h)}} \deg(P_q)$ for every $1 \le i  \le r_k$ and $1 \le q \le n-k$.
  \end{lema}
  
  \begin{proof}
  It is clear by definition and (\ref{lines}) that $\delta_q(W_i^{(k)}) \lesssim_{B^{(h)}} \deg(P_q)$. Assume for contradiction there is some $q$ with $\delta_q(W_i^{(k)}) \le \epsilon \deg(P_q)$ for some sufficiently small $\epsilon \gtrsim_{B^{(h)}} 1$. By definition of the quantities $\delta_j(W_i^{(k)})$ it follows from Lemma \ref{simplelinear} and Bezout's theorem that 
  $$\deg(W_i^{(k)}) \le \prod_{j=1}^{n-k} \delta_j(W_i^{(k)}) \lesssim_{B^{(h)}} \epsilon \prod_{j=1}^{n-k} \deg(P_j).$$
  If $\epsilon$ is sufficiently small with respect to $B^{(h)}$ this contradicts that, by (\ref{lines}) and Bezout's theorem, it is $\deg(W_i^{(k)}) \sim_{B^{(h)}} \prod_{j=1}^{n-k} \deg(P_j)$.
  \end{proof}
 
   The procedure described in \S \ref{OM} suggests we will be using induction on the level of reduction, with a stronger bound the smallest the level of reduction of $T$ is. Indeed, we will actually prove the following equivalent form of Theorem \ref{I0}.
   
\begin{teo}
\label{I2}
Let $\K$ be an algebraically closed field. Given $1 \le l \le n$, there exist $K_{l,1} < K_{l,2} < \ldots < K_{l,n-\dim(S)} \lesssim_n 1$ such that if $T$ is a set of $l$-dimensional irreducible varieties in $\K^n$ that reduces strongly to level $h$, with $l \le h \le n$, then for every set of varieties $S$ with $\dim(S)+1=\dim(T)$, we have
$$ \mathcal{I}(S,T) \le \sum_{1 \le m \le n-\dim(S)} K_{l,m}^{(h)} \deg(S)^{1/m} \deg(T)^{1-1/m} \Dc_{m+\dim(S)}(T)^{1/m},$$
with $K_{l,m}^{(h)}=K_{l,h-\dim(S)}$ if $m \ge h-\dim(S)$ and $K_{l,m}^{(h)} = K_{l,m}$ otherwise.
\end{teo}

The case $l=n$ of this result is trivial. Notice that by definition every finite family $T$ reduces strongly to some level. Also, if $T$ reduces to level $l$, this necessarily means that it has $O_n(1)$ elements and therefore the number of incidences is clearly bounded by $O_n(\deg(S))$. We may therefore proceed by induction both on the size of $T$ and on its level. Precisely, we pick a set $T$ that reduces strongly to level $h > l$ and assume the result holds for all proper subsets of $T$ and that we already know the result holds if $T$ reduces strongly to level $h-1$.

\begin{rmk}
For $|T|=O_n(1)$ we could alternatively have observed that if $S$ is $k$-free with respect to $T$ then by Bezout's inequality
$$ \mathcal{I}(S,T) \le \deg(S) + O_n \left( \min \left\{  k, \max_{t \in T} \deg(t)^2 \right\} \right) \le \deg(S) + O_n ( k^{1/2} \deg(T) ),$$
and use this instead of the bound $O_n(\deg(S))$. An easy inspection of the arguments below shows that that this would then lead to a version of Theorem \ref{I2}, and therefore Theorem \ref{I0}, that replaces the first term of the bound with this expression. This, in turn, can be used to deduce Theorem \ref{I1}.
\end{rmk}

\subsection{Diminishing the concentration}
\label{32}

From now on we fix a choice of $T$ reducing strongly to level $h$. We will use the notations of Definition \ref{reduction} and Definition \ref{SR} to refer to the corresponding objects associated to our fixed choice of $T$ and also abbreviate $\mathcal{D}_k(T)$ as $\mathcal{D}_k$. We will simply say $T$ admits a good partition at level $k+1$ to mean that it admits a good partition over $W^{(k+1)}$.

Our goal in this section is to prune $T$ via a polynomial $f$ of adequate degree in such a way that $T \setminus T_f$ is a large subset of $T$ for which we have a good amount of control on the genericity of its subsets.

 We begin with the following simple observation.
 
 \begin{lema}
 \label{DHI}
For every $h \le k < n$, we have
 $$ \frac{\deg(W^{(k)})}{\deg(W^{(k+1)})} \lesssim_{B^{(h)}} \deg_R(T, W^{(k+1)}).$$
 \end{lema}
 
 \begin{proof}
 This follows from (\ref{hmin}), the definition of $\deg_R(T,W^{(k+1)})$ and Bezout's theorem.
 \end{proof}

To achieve our goal, we will need to study the quantities
$$ D_k^{(A,B)}(T) = \max \left\{ \frac{\deg(T_W)}{\deg(W)} : \dim(W)=k , \, A < \deg(W) < B \right\},$$
for certain choices of real numbers $0 < A < B$ and where the maximum goes along all $k$-dimensional algebraic sets whose degree lies in the range $(A,B)$. 

For every $h \le k \le n$, we will consider the parameters 
\begin{equation}
\begin{aligned}
R_1^{(k)} &= \varepsilon_1 \deg(W^{(k)}) \\
R_2^{(k)} &= \varepsilon_2 \deg(W^{(k)})
\end{aligned}
\end{equation}
 for some $\varepsilon_1,\varepsilon_2 \sim_{B^{(h)}} 1$ to be specified with $\varepsilon_2 > 2 \varepsilon_1$ and study in particular the corresponding quantity $D_k^{(R_1^{(k)},R_2^{(k)})}(T)$. 
 
The following lemma shows that a $k$-dimensional variety of degree much smaller than $\deg(W^{(k)})$ cannot contain too many elements from $T$.

\begin{lema}
\label{SI}
Let $h \le k < n$. If $W$ is a $k$-dimensional algebraic set with $\deg(W) < R_2^{(k)}$ and $\varepsilon_2 > 0$ is chosen sufficiently small with respect to $B^{(h)}$, then $\deg(T_W) < a_{k+1}^{(h)} \deg(T)$.
\end{lema}

\begin{proof}
By Bezout's theorem, we have that
\begin{equation}
\begin{aligned}
 \deg(W) &< \varepsilon_2 \deg(W^{(k)}) \\
  &\le \varepsilon_2 \deg(W^{(k+1)}) \deg(P_{n-k}). \\
 \end{aligned}
 \end{equation}
 It then follows from Lemma \ref{BL} and the fact that 
 $$\deg(P_{n-k}) \gtrsim_{B^{(h)}} \deg(P_{n-k-1}) \gtrsim_{B^{(h)}} \delta(W_i^{(k+1)}),$$
 for every $1 \le i \le r_{k+1}$, that we can find a polynomial $f$ with 
 \begin{equation}
 \label{degf}
 \deg(f) \lesssim_{B^{(h)}}  \varepsilon_2^{1/n} \deg(P_{n-k}) \lesssim_{B^{(h)}} \varepsilon_2^{1/n} \frac{\deg(W^{(k)})}{\deg(W^{(k+1)})}, 
 \end{equation}
 vanishing on $W$ and such that $Z(f) \cap W^{(k+1)}$ is $k$-dimensional and has degree at most
  $$\lesssim_{B^{(h)}} \varepsilon_2^{1/n} \deg(W^{(k)}).$$
  Notice that in (\ref{degf}) we are using one of the estimates from (\ref{lines}). If we choose $\varepsilon_2>0$ sufficiently small with respect to $B^{(h)}$ we deduce from (\ref{hmin}) that $f$ cannot vanish at all of $T$. But notice also that by (\ref{degf}) and Lemma \ref{DHI} the degree of $f$ can be made strictly less than $\deg_R(T,W^{(k+1)})$, again provided $\varepsilon_2>0$ is chosen sufficiently small with respect to $B^{(h)}$. The bound on $T_W$ then follows from the fact that $T$ reduces to level $h$ and therefore does not admit an $a_{k+1}^{(h)}$-good partition at level $k+1$.
 \end{proof}

We can now formalise the idea motivating the definition of strong reductions.
  
  \begin{lema}
  \label{srl}
  If $T$ reduces strongly to level $h$ then every subset $\tilde{T} \subseteq T$ with $\deg(\tilde{T}) \ge \deg(T)/2$ reduces to level $h$ with respect to the same polynomials $P_{n-k}$ and varieties $W^{(k)}$ that $T$ reduces strongly to level $h$.
  \end{lema}
    
  \begin{proof}
  The only non-trivial statement is (\ref{hmin}) and this follows from Lemma \ref{SI}, provided $b_0^{(h)}$ is chosen sufficiently small with respect to the other parameters in $B^{(h)}$.
  \end{proof}
  
    We have the following consequence.

 \begin{lema}
 \label{DHI2}
Let $\tilde{T} \subseteq T$ with $\deg(\tilde{T}) \ge \deg(T)/2$. For every $h \le k < n$, we have
 $$ \frac{\deg(W^{(k)})}{\deg(W^{(k+1)})} \lesssim_{B^{(h)}} \deg_R(\tilde{T}, W^{(k+1)}).$$
 \end{lema}
 
 \begin{proof}
 By Lemma \ref{srl} we can use the same reasoning as in Lemma \ref{DHI}.
 \end{proof}
 
Lemma \ref{SI} easily implies the following bound on $\Dc_k^{(R_1^{(k)},R_2^{(k)})}(T)$.

\begin{lema}
\label{MR}
If $\varepsilon_1 > 2 a_{h+1}^{(h)}$ and $\varepsilon_2$ is as in the statement of Lemma \ref{SI}, then $\Dc_k^{(R_1^{(k)},R_2^{(k)})}(T) < \Dc_k/2$ for every $h \le k \le n$.
\end{lema}

\begin{proof}
Suppose the result fails. Then we can find some $W$ of dimension $k$ with $R_1^{(k)} < \deg(W) < R_2^{(k)}$ such that $\frac{\deg(T_W) }{\deg(W)} \ge \Dc_k/2$. In particular, we have that
\begin{equation}
\begin{aligned}
 \deg(T_W) &\ge \frac{1}{2} \deg(T) \frac{\deg(W)}{\deg(W^{(k)})} \\
 &> \frac{\varepsilon_1}{2}  \deg(T) \\
 &> a_{h+1}^{(h)} \deg(T).
 \end{aligned}
 \end{equation}
Since $a_{h+1}^{(h)} \ge a_{k+1}^{(h)}$, this contradicts Lemma \ref{SI}.
\end{proof}
  
  Given a large subset $\tilde{T}$ of $T$, we want to be able to find a polynomial $f$ of adequate degree such that $\tilde{T} \setminus \tilde{T}_f$ is a large subset of $T$ that is significantly less-concentrated in varieties of small degree. This will be done in the most straightforward way, by simply removing all subsets $\tilde{T}'$ of $\tilde{T}$ that lie inside varieties that have abnormally small degree with respect to $\deg(\tilde{T}')$. In this sense, this part of the method is quite similar to the pruning mechanisms employed in \cite{G5,GZ}. As we mentioned, in our case we also need to ensure that all these elements we are removing lie inside of $Z(f)$ for some $f$ of adequate degree. This will be accomplished using the arguments and lemmas that we have established so far in this section.
  
 \begin{lema}
 \label{T0}
Let $\tilde{T} \subseteq T$ with $\deg(\tilde{T}) \ge \deg(T)/2$. If $\varepsilon_2>0$ is sufficiently small with respect to $B^{(h)}$, there exists a polynomial $f$ with $\deg(f) < \deg_R(\tilde{T},W^{(h+1)})$ such that $\deg(\tilde{T}_f)< \sum_{k=h}^{n-1} a_{k+1}^{(h)} \deg(T)$ and $\Dc_k^{(1,R_2^{(k)})}(\tilde{T} \setminus \tilde{T}_f) < \Dc_k/2$ for every $h \le k \le n$.
  \end{lema}
 
 \begin{proof}
 The result is trivial if $h=n$, so we may assume $h<n$. Let $\tilde{T}^{(h-1)}:=\tilde{T}$. We will recursively construct subsets $\tilde{T}^{(h)} \supseteq \cdots \supseteq \tilde{T}^{(n-1)}$ of $\tilde{T}$ in the following way. Given $h \le k < n$, if $\Dc_k^{(1,R_2^{(k)})}(\tilde{T}^{(k-1)}) < \Dc_k/2$ we simply take $\tilde{T}^{(k)} := \tilde{T}^{(k-1)}$. Otherwise, we shall recursively construct a finite family of algebraic sets $Z_1^{(k)},\ldots,Z_r^{(k)}$ as follows. We let $Z_1^{(k)}$ be an irreducible algebraic set of dimension $k$ with $\deg(Z_1^{(k)}) \le R_2^{(k)}$ and $\deg(\tilde{T}_{Z_1^{(k)}}) \ge \frac{\Dc_k}{2} \deg(Z_1^{(k)})$, which we are assuming exists. Recursively, suppose we have constructed sets $Z_1^{(k)},\ldots,Z_j^{(k)}$. If $\Dc_k^{(1,R_2^{(k)})}(\tilde{T}^{(k-1)} \setminus \bigcup_{i=1}^j \tilde{T}_{Z_i^{(k)}}) < \Dc_k/2$ we halt the process. Otherwise, we let $Z_{j+1}^{(k)}$ be an irreducible algebraic set with $\deg(Z_{j+1}^{(k)}) \le R_2^{(k)}$ and such that the total degree of the elements of $\tilde{T}^{(k-1)} \setminus \bigcup_{i=1}^j \tilde{T}_{Z_i^{(k)}}$ contained in $Z_{j+1}^{(k)}$ is at least $\frac{\Dc_k}{2} \deg(Z_{j+1}^{(k)})$. 
 
 Let $\varepsilon_1$ be as in the statement of Lemma \ref{MR} and assume also that it is chosen sufficiently small with respect to $B^{(h)}$ and $\varepsilon_2$. By Lemma \ref{MR} it must be $\deg(Z_j^{(k)}) \le R_1^{(k)}$ for every $1 \le j \le r$. Write $Z^{(k)}=Z_1^{(k)} \cup \ldots \cup Z_r^{(k)}$ and assume $\deg(Z^{(k)}) > R_1^{(k)}$. Then, there exists some $r'$ with 
 $$\deg(Z_1^{(k)} \cup \ldots \cup Z_{r'-1}^{(k)}) \le R_1^{(k)} <  \deg(Z_1^{(k)} \cup \ldots \cup Z_{r'}^{(k)}) \le 2R_1^{(k)}<R_2^{(k)}.$$
 However, we know by construction that the total degree of elements of $\tilde{T}$ in $Z_{\ast}^{(k)} = Z_1^{(k)} \cup \ldots \cup Z_{r'}^{(k)}$ is at least $\Dc_k \deg(Z_{\ast}^{(k)})/2$ and this would contradict Lemma \ref{MR}. Therefore, we conclude that it must be $\deg(Z^{(k)}) \le R_1^{(k)}$. 
 
 The same argument as in Lemma \ref{SI} now gives us a polynomial $f_k$ of degree at most $\lesssim_{B^{(h)}} \varepsilon_1^{1/n} \deg(P_{n-k})$ that vanishes on $Z^{(k)}$ without vanishing on any component of $W^{(k+1)}$. In particular, we have that $A^{(k)}=Z(f_k) \cap W^{(k+1)}$ satisfies the hypothesis of Lemma \ref{SI}, provided $\varepsilon_1$ was chosen sufficiently small with respect to $\varepsilon_2$ and $B^{(h)}$, and therefore $\deg(\tilde{T}_{A^{(k)}}) < a_{k+1}^{(h)} \deg(T)$. 
 
 We let $\tilde{T}^{(k)} := \tilde{T}^{(k-1)} \setminus \tilde{T}_{A^{(k)}}$. Once we have constructed $\tilde{T}^{(h)}, \ldots,\tilde{T}^{(n-1)}$ in this way, we let $\tilde{T}_1 = \tilde{T}^{(n-1)}$ and write $\tilde{T}_0=\tilde{T} \setminus \tilde{T}_1$. It is clear from construction that $\deg(\tilde{T}_0) < \sum_{k=h}^{n-1} a_{k+1}^{(h)} \deg(T)$. Furthermore, the polynomial $f = \prod_{k=h}^n f_k$ satisfies $\tilde{T}_f=\tilde{T}_0$ and
  \begin{equation}
 \begin{aligned}
  \deg(f) &\lesssim_{B^{(h)}} \varepsilon_1^{1/n} \sum_{k=h}^n \deg(P_{n-k}) \\
  &\lesssim_{B^{(h)}} \varepsilon_1^{1/n} \deg(P_{n-h}) \\
  &< \deg_R(\tilde{T},W^{(h+1)}) ,
    \end{aligned}
  \end{equation}
  upon choosing $\varepsilon_1>0$ sufficiently small with respect to $B^{(h)}$, where we are using (\ref{lines}) and Lemma \ref{DHI2}. The remaining claim on $\tilde{T}_1=\tilde{T} \setminus \tilde{T}_f$ is also clear by construction.
 \end{proof}
 
 An immediate consequence of the last lemma is that small subsets of $\tilde{T} \setminus \tilde{T}_f$ are significantly less concentrated than $T$.
 
 \begin{coro}
 \label{DH2}
 Let the notation be as in Lemma \ref{T0} and let $T'$ be a subset of $\tilde{T} \setminus \tilde{T}_f$ with $\deg(T') < \frac{\varepsilon_2}{2} \deg(T)$. Then $\Dc_k(T') < \Dc_k/2$ for every $h \le k \le n$.
 \end{coro}
 
 \begin{proof}
 Since $T'$ is a subset of $\tilde{T} \setminus \tilde{T}_f$ and $\Dc_k^{(1,R_2^{(k)})}(\tilde{T} \setminus \tilde{T}_f) < \Dc_k/2$ for every $h \le k \le n$ by Lemma \ref{T0}, we only need to check what happens with those $k$-dimensional algebraic sets $W$ with $\deg(W) \ge R_2^{(k)}=\varepsilon_2 \deg(W^{(k)})$. But then we have
 $$ \frac{\deg(T'_W)}{\deg(W)} \le \frac{\deg(T')}{\deg(W)} < \frac{\varepsilon_2}{2} \frac{\deg(T)}{\deg(W)} \le \frac{1}{2} \frac{\deg(T)}{\deg(W^{(k)})} \le \Dc_k/2,$$
 as desired.
 \end{proof}
  
  \subsection{Ordering of the relative degrees}
  \label{33}
  
  We have the following analogue of Lemma \ref{SI} for $(h-1)$-dimensional varieties.
  
    \begin{lema}
    \label{loweb}
  Let $Y$ be an $(h-1)$-dimensional set containing $a_{h+1}^{(h)} \deg(T)$ elements of $T$. Then $\deg(Y) \gtrsim_{B^{(h)}} \deg(W^{(h)}) \deg(P_{n-h})$.
  \end{lema}
 
 \begin{proof}
 Assume $\deg(Y)  \le \epsilon \deg(W^{(h)}) \deg(P_{n-h})$ for some small $\epsilon > 0$. It follows from Bezout's theorem that
 $$ \deg(Y) \le \epsilon \deg(W^{(h+1)}) \deg(P_{n-h})^2,$$
 and we can therefore apply Lemma \ref{BL}, (\ref{lines}) and Lemma $\ref{DHI}$ to find a polynomial $f$ of degree
 $$\lesssim_{n} \epsilon^{1/n} \deg(P_{n-h}) \lesssim_{B^{(h)}} \epsilon^{1/n} \frac{\deg(W^{(h)})}{\deg(W^{(h+1)})} \lesssim_{B^{(h)}} \epsilon^{1/n} \deg_R(T,W^{(h+1)}),$$
  vanishing on $Y$ without vanishing identically on any component of $W^{(h+1)}$. If $\epsilon>0$ is sufficiently small with respect to $B^{(h)}$ it follows from the definition of $\deg_R(T,W^{(h+1)})$ that $f$ cannot vanish on all of $T$. We conclude that $f$ produces an $a_{h+1}^{(h)}$-good partition of $T$ at level $h+1$, which contradicts the fact that $T$ reduces to level $h$.
 \end{proof}

This last lemma allows us to establish the correct ordering for all the relative degrees we shall need in the proof.

\begin{coro}
\label{or}
For every subset $\tilde{T} \subseteq T$ with $\deg(\tilde{T}) \ge \deg(T)/2$ and every $h \le k < n$, we have
$$ \deg_R(\tilde{T},W^{(k+1)}) \lesssim_{B^{(h)}} \deg_R(\tilde{T},W^{(k)}).$$
\end{coro}

\begin{proof}
Let $h \le k < n$. From Definition \ref{reduction} it is clear that $P_{n-k}$ vanishes on $\tilde{T}$ without vanishing on any irreducible component of $W^{(k+1)}$, so in particular $\deg_R(\tilde{T},W^{(k+1)}) \le \deg(P_{n-k})$.
On the other hand, it follows from (\ref{lines}) and Lemma \ref{DHI2} that
$$ \deg_R(\tilde{T},W^{(k+1)}) \gtrsim_{B^{(h)}}  \frac{\deg(W^{(k)})}{\deg(W^{(k+1)})} \gtrsim_{B^{(h)}} \deg(P_{n-k}).$$
Therefore $\deg_R(\tilde{T},W^{(k+1)}) \sim_{B^{(h)}} \deg(P_{n-k})$. The result now follows from (\ref{lines}) in the range $h < k < n$ and from Lemma \ref{loweb} and Bezout's theorem when $k=h$.
\end{proof}

Notice that this shows that the bound on the degree of the polynomial obtained in Lemma \ref{T0} is consistent with the discussion in \S \ref{OM}.
 
 \section{Relative partitions}
 \label{4}
 
 Throughout this section we will fix a choice of $T$ reducing strongly to some level $l < h \le n$ and use the notations of Definition \ref{reduction} and Definition \ref{SR} to refer to the corresponding objects associated with $T$. In particular, we shall say we have a good partition at level $k+1$ to mean we have a good partition over $W^{(k+1)}$.

 To prove Theorem \ref{I0} we will need to obtain partitions of subsets $L$ of $T$ that are good with respect to a fixed subset $\tilde{T}$ of $T$. After setting up this context in \S \ref{41}, we will show in \S \ref{42} that such a partition can always be guaranteed if we can cover $L$ with a few small $(h-1)$-dimensional varieties. We will then use this in \S \ref{43} to show that a set not admitting a good partition reduces strongly to level $h-1$ after possibly eliminating some elements lying in the zero set of a polynomial of low degree. Finally, in \S \ref{44} we show how to construct an adequate partition satisfying a weak form of the dichotomy discussed in \S \ref{OM}. 
  
  \subsection{Working over a subset}
  \label{41}
  
 From now on, we fix some $\tilde{T} \subseteq T$ with $\deg(\tilde{T}) \ge \deg(T)/2$, so in particular $\tilde{T}$ satisfies the conclusions of Lemma \ref{srl}. Given any subset $L \subseteq \tilde{T}$, we shall write $Z_L$ for an $(h-1)$-dimensional algebraic set containing $L$ of the smallest possible degree. We will build partitions of $L$ that provide convenient partitions for $\tilde{T}$. Because of this, will find the following definition useful.
 
  \begin{defi}
 Given $h \le k \le n$, we say $L$ admits a $\tau$-good partition at level $k$ with respect to $\tilde{T}$ if there exists a polynomial $f$ of degree at most $\deg_R(\tilde{T},W^{(k)})$ such that 
 $$\tau \deg(L) < \deg(L_{f}) < \deg(L).$$
 \end{defi}
 
 We emphasise that the difference is that in this definition we require a bound for $\deg(f)$ of the form $\deg_R(\tilde{T},W^{(k)})$ instead of depending on the relative degree of $L$ itself.
 
  Notice that if $\deg(L) \ge a_{h+1}^{(h)} \deg(T)$ we have the lower bound for $\deg(Z_L)$ given by Lemma \ref{loweb}. On the other hand, the following lemma is clear from Definition \ref{RD} and Bezout's theorem.
 
 \begin{lema}
 \label{ZL}
  For any subset $L \subseteq \tilde{T}$, we have
 $$ \frac{\deg(Z_L)}{\deg(W^{(h)})}  \le \deg_R(\tilde{T},W^{(h)}).$$
 \end{lema}
 
  \subsection{Obtaining a good partition}
  \label{42}
 
Let $L$ be a subset of $\tilde{T}$ whose degree is not too small with respect to $\deg(\tilde{T})$. We are going to prove that we can find a good partition of $L$ at level $h$ with respect to $\tilde{T}$ as long as we can cover a large part of it with a few small $(h-1)$-dimensional varieties. This will later be used to obtain some structure on $L$ when we know such a partition fails to exist.
 
 \begin{lema}
  \label{SV}
 Let $L \subseteq \tilde{T}$ be a subset of $\tilde{T}$ with $\deg(L) \ge a_{h+1}^{(h)} \deg(T)$. Let $U_1,\ldots,U_s$ be $(h-1)$-dimensional irreducible algebraic varieties and write $U = \bigcup_i U_i$. Suppose $\deg(U) \le C \deg(Z_L)$, $\deg(L_U) \ge \varepsilon_3 \deg(L)$ for some absolute constants $C,\varepsilon_3 > 0$ and that $\deg(U_i) \le \varepsilon_4 \deg(Z_L)$ for every $1 \le i \le s$ and for some  absolute constant $0 < \varepsilon_4 < C$ that is sufficiently small with respect to $B^{(h)}$. Then $L$ admits an $\varepsilon_3 \varepsilon_4 C^{-1}$-good partition at level $h$ with respect to $\tilde{T}$.
  \end{lema}
  
  \begin{proof}
 Let $U_1'$ be the irreducible component of $U$ that maximises the expression $\frac{\deg(L_{U_1'})}{\deg(U_1')}$ and write $L^{(1)}=L_{U_1'}$. Recursively, if we have constructed $U_1',\ldots,U_r'$ and $L^{(1)},\ldots,L^{(r)}$, we let $U_{r+1}'$ be component of  $U$ that maximises the expression 
 $$\frac{\deg((L \setminus \bigcup_{i=1}^r L^{(i)})_{U_{r+1}'})}{\deg(U_{r+1}')},$$
 among all remaining components of $U$ different from $U_1',\ldots,U_r'$. We then write $L^{(r+1)}$ for those elements of $L$ inside of $U_{r+1}'$ that do not lie in $U_j'$ for any $j < r+1$ (so $L^{(r+1)}=(L \setminus \bigcup_{i=1}^r L^{(i)})_{U_{r+1}'}$). After reordering the original $U_1,\ldots,U_s$ if necessary, to alleviate notation we may simply assume $U_i=U_i'$ for every $1 \le i \le s$. Notice however that this new order guarantees that
 $$ \frac{\sum_{i \le j} \deg(L^{(i)})}{\deg(L)} \ge \varepsilon_3 \frac{\sum_{i \le j} \deg(U_i)}{\deg(U)},$$
 for every $1 \le j \le s$.
  
  Observe now that either $ \deg(U) \le \varepsilon_4 \deg(Z_L)$ or there is a minimal $t$ such that the sum of the degrees of $U_1,\ldots,U_t$ exceeds $\varepsilon_4 \deg(Z_L)$. This last option in turn means on the one hand that the sum of the degrees of $U_1,\ldots,U_t$ is at most $2\varepsilon_4 \deg(Z_L)$, but it also means by our ordering that 
  \begin{equation}
  \label{lie}
   \sum_{1 \le i \le t} \deg(L^{(i)}) \ge \frac{\varepsilon_3 \varepsilon_4}{C} \deg(L).
   \end{equation}
  So in either case we end up with a collection $U_1,\ldots,U_t$ with their degrees summing to at most $2 \varepsilon_4 \deg(Z_L)$ and satisfying (\ref{lie}). Write $U^{\ast}=U_1 \cup \ldots \cup U_t$. It follows from Lemma \ref{BL} and Lemma \ref{loweb} that we can find a polynomial $f$ with 
  \begin{equation}
  \label{CB}
  \deg(f) \lesssim_{B^{(h)}} \varepsilon_4^{1/n} \frac{\deg(Z_L)}{\deg(W^{(h)})},
  \end{equation}
  vanishing on $U^{\ast}$ and cutting $W_i^{(h)}$ properly, for every $1 \le i \le r_h$. Choosing $\varepsilon_4>0$ sufficiently small with respect to $B^{(h)}$ we conclude, from the minimality of $Z_L$ and Bezout's theorem, that $Z(f)$ cannot contain all of $L$. Furthermore, again choosing $\varepsilon_4>0$ sufficiently small with respect to $B^{(h)}$, we see from Lemma \ref{ZL} that we can additionally guarantee that $\deg(f)$ is bounded by $\deg_R(\tilde{T},W^{(h)})$. Thus, $f$ provides the desired $\varepsilon_3 \varepsilon_4 C^{-1}$-good partition of $L$ at level $h$ with respect to $\tilde{T}$. 
  \end{proof}
  
  \subsection{Decreasing the level}
  \label{43}
  
  We will now show that if $L$ is a large subset of $\tilde{T}$ that does not admit a good partition then, after removing some elements from $L$ lying in the zero set of some small polynomial $f$, we can obtain a large subset of $L$ that reduces strongly to level $h-1$.
  
  \begin{prop}
  \label{nogood}
 Given any $\varepsilon_3 \gtrsim_{B^{(h)}} 1$ there exists some $\tau \gtrsim_{B^{(h)},a_{h}^{(h-1)}}1$ such that the following holds. Let $L$ be a subset of $\tilde{T}$ with $\deg(L) \gtrsim_{B^{(h)}} \deg(T)$. If $L$ does not admit a $\tau$-good partition at level $h$ with respect to $\tilde{T}$, then we can find a (possibly trivial) partition $L=L_1 \cup L_2$ such that $\deg(L_2) \ge (1-\varepsilon_3) \deg(L)$, $L_2$ reduces strongly to level $h-1$ and $L_1=L_f$ for some polynomial $f$ of degree $\lesssim_{B^{(h)}} \deg_R(\tilde{T},W^{(h)}) $.
  \end{prop}
  
  \begin{proof}
  Let $P_{n-h+1}$ be a polynomial of the smallest possible degree vanishing on $L$ without vanishing on any component of $W^{(h)}$ and let
  $$ Z(P_1,\ldots,P_{n-h+1}) = Z_1 \cup \ldots \cup Z_s,$$
  be the irreducible decomposition of $Z(P_1,\ldots,P_{n-h+1})$. Notice that every irreducible component $Z_j$ of $Z(P_1,\ldots,P_{n-h+1})$ that has dimension at least $h$ must be contained inside an irreducible component of $Z(P_1,\ldots,P_{n-h})$ that is not of the form $W_i^{(h)}$ for some $1 \le i \le r_h$. Therefore, such a choice of $Z_j$ cannot contain elements from $L$, implying that all elements of $L$ are contained inside the $(h-1)$-dimensional irreducible components of $Z(P_1,\ldots,P_{n-h+1})$.
  
  By (\ref{lines}) and Lemma \ref{loweb} we know that 
  $$\deg(Z_L) \gtrsim_{B^{(h)}} \deg(W_i^{(h)}) \delta(W_i^{(h)}),$$
  for every $1 \le i \le r_h$ and therefore, by Lemma \ref{BL}, we can find some polynomial of degree $\lesssim_{B^{(h)}} \frac{\deg(Z_L)}{\deg(W^{(h)})}$ vanishing on $Z_L$ (and therefore on all of $L$) without vanishing on any component of $W^{(h)}$. By the minimality of $\deg(P_{n-h+1})$, this means that
  \begin{equation}
  \label{owl}
  \deg(P_{n-h+1}) \lesssim_{B^{(h)}} \frac{\deg(Z_L)}{\deg(W^{(h)})}.
  \end{equation}
  In particular, since $\deg(W^{(h)}) \gtrsim_{B^{(h)}} \deg(Z(P_1,\ldots,P_{n-h}))$ by Definition \ref{reduction}, we see from Bezout's theorem that 
  \begin{equation}
  \label{FF}
  \deg(Z(P_1,\ldots,P_{n-h+1})) \lesssim_{B^{(h)}} \deg(Z_L).
  \end{equation}
  
  We may assume the $Z_i$ are ordered in increasing order of dimension and then on decreasing order of degrees. Let $\epsilon_1 \gtrsim_{B^{(h)}} 1$ be a parameter to be specified soon. We write $Z^{(1)}=Z_1 \cup \ldots \cup Z_t$ for the union of those irreducible components of dimension $h-1$ with $\deg(Z_i) > \epsilon_1 \deg(Z_L)$, $Z^{(2)}=Z_{t+1} \cup \ldots \cup Z_r$ for the union of those irreducible components of $Z$ of dimension $h-1$ and degree at most $\epsilon_1 \deg(Z_L)$ and $Z^{(3)}=Z_{r+1} \cup \ldots \cup Z_s$ for the remaining ones. In particular, by a previous observation we know that $Z^{(3)}$ does not contain any element from $L$. 
  
  Notice that we may group the components of $Z^{(2)}$ into at most $\lesssim \epsilon_1^{-1}$ algebraic sets $Z_1^{(2)}, \ldots, Z_m^{(2)}$ of degree at most $\epsilon_1 \deg(Z_L)$. For each such $Z_i^{(2)}$ we can use Lemma \ref{BL} to find a polynomial $Q_i$ of degree $\lesssim_{B^{(h)}} \epsilon_1^{1/n} \frac{\deg(Z_L)}{\deg(W^{(h)})}$ vanishing on $Z_i^{(2)}$ without vanishing on any component of $W^{(h)}$. In particular, upon choosing $\epsilon_1$ sufficiently small with respect to $B^{(h)}$, we can guarantee that
  $$ \deg(Z(Q_i) \cap W^{(h)}) \le \deg(Q_i) \deg(W^{(h)}) \le \varepsilon_4 \deg(Z_L),$$
  for some $\varepsilon_4 \sim_{B^{(h)}} 1$ that is sufficiently small with respect to $B^{(h)}$ as required in the statement of Lemma \ref{SV}. We conclude that the polynomial $Q=\prod_{1 \le i \le m} Q_i$ has degree 
  \begin{equation}
  \label{CI}
  \lesssim_{B^{(h)}} \epsilon_1 ^{-1} \varepsilon_4 \frac{\deg(Z_L) }{\deg(W^{(h)})},
  \end{equation}
  and all irreducible components of $U=Z(Q) \cap W^{(h)}$ have dimension $h-1$ and degree at most $\varepsilon_4 \deg(Z_L)$. Furthermore, we see from Bezout's theorem that $ \deg(U) \le C \deg(Z_L)$ for some $C \lesssim_{B^{(h)}} 1$. Given that we also have the estimate $\varepsilon_4 \gtrsim_{B^{(h)}} 1$, upon choosing $\tau$ sufficiently small with respect to $\varepsilon_3$ and $B^{(h)}$ and since we are assuming $L$ does not admit a $\tau$-good partition at level $h$ with respect to $\tilde{T}$, it follows from Lemma \ref{SV} that $U$ can contain at most $\varepsilon_3 \deg(L)$ elements of $L$.
  
   We write $L_1$ for those elements of $L$ inside of $U$ and $L_2=L \setminus L_1$, so in particular no element from $L_2$ lies in $Z^{(2)}$ or $Z^{(3)}$. Since by Lemma \ref{ZL} the expression (\ref{CI}) is $\lesssim_{B^{(h)}} \deg_R(\tilde{T},W^{(h)})$, the result will follow as long as we can show that $L_2$ reduces strongly to level $h-1$. We will show this is the case with $W^{(h-1)}$ being the algebraic set obtained from $Z^{(1)}$ upon discarding those irreducible components that contain no element from $L_2$. Notice that this guarantees that $W^{(h-1)} \subseteq W^{(h)}$, since we already know that $W^{(h-1)} \subseteq Z(P_1,\ldots,P_{n-h})$ and no irreducible component of the latter algebraic set, other than the irreducible components of $W^{(h)}$, contain elements from $L$. Furthermore, we have also ensured that no irreducible component of $Z(P_1,\ldots,P_{n-h+1})$, other than those of $W^{(h-1)}$, contain any element from $L_2$.
   
   Consider now $Z_{L_2}$ and assume $\deg(Z_{L_2}) < \epsilon_2 \deg(Z_L)$. If $\epsilon_2>0$ is sufficiently small with respect to $B^{(h)}$, we can use Lemma \ref{BL}, Lemma \ref{loweb} and Lemma \ref{ZL} to find a polynomial $f$ of degree at most 
   $$ \lesssim_{B^{(h)}} \epsilon_2^{1/n} \frac{\deg(Z_L)}{\deg(W^{(h)})} < \deg_R(\tilde{T},W^{(h)}),$$
   that vanishes on $Z_{L_2}$ (and therefore on $L_2$) and with $Z(f) \cap W^{(h)}$ having dimension $h-1$ and degree strictly less than $\deg(Z_L)$ by Bezout's theorem. Given that $\deg(L_2) \ge (1-\varepsilon_3) \deg(L)$, this would mean that this polynomial gives a $\tau$-good partition of $L$ at level $h$ with respect to $\tilde{T}$, which we are assuming it is not possible. We have thus shown that 
  \begin{equation}
  \begin{aligned}
  \deg(Z_{L_2}) &\gtrsim_{B^{(h)}} \deg(Z_L) \\
  &\gtrsim_{B^{(h)}} Z(P_1,\ldots,P_{n-h+1}) \\
   &\gtrsim_{B^{(h)}} \deg(W^{(h-1)}),
   \end{aligned}
   \end{equation}
  by (\ref{FF}), giving us the estimate (\ref{barmin}) for $W^{(h-1)}$. Notice that this estimate for $W^{(k)}$, $h \le k < n$, is a consequence of Lemma \ref{SI}.
  
    Suppose $L_2$ admits an $a_{h}^{(h-1)}/2$-good partition at level $h$. This means we can find a polynomial $f$ of degree at most $\deg_R(L_2,W^{(h)}) \le \deg_R(\tilde{T},W^{(h)})$, not vanishing on $L_2 \subseteq L$ and with $Z(f)$ containing at least $a_{h}^{(h-1)} \deg(L_2)/2 \ge \tau \deg(L)$ elements from $L_2$, provided $\tau$ was chosen sufficiently small with respect to $a_{h}^{(h-1)}$. This gives a $\tau$-good partition of $L$ at level $h$ with respect to $\tilde{T}$, which we are assuming is not possible. 
  
  Similarly, for every $h \le k < n$, we know $L_2$ does not admit an $a_{k+1}^{(h-1)}$/2-good partition at level $k+1$ from the fact that $\tilde{T}$ does not admit an $a_{k+1}^{(h)}$-good partition at level $k+1$. Indeed, suppose $f$ is a polynomial giving an $a_{k+1}^{(h-1)}/2$-good partition of $L_2$ at level $k+1$, for some $h \le k < n$. This means that $f$ is a polynomial of degree at most $\deg_R(L_2,W^{(k+1)}) \le \deg_R(\tilde{T},W^{(k+1)})$, that does not vanish at all elements of $L_2 \subseteq \tilde{T}$ and such that $Z(f)$ contains at least $a_{k+1}^{(h-1)} \deg(L_2)/2 \gtrsim_{B^{(h)}} a_{k+1}^{(h-1)} \deg(T)$ elements from $L_2$. This contradicts the fact that $\tilde{T}$ does not admit an $a_{k+1}^{(h)}$-good partition at level $k+1$, provided $a_{k+1}^{(h)}$ was chosen sufficiently small with respect to $a_{k+1}^{(h-1)}$ and $B^{(h)}$.
   
 Recall that we have defined the irreducible components $W_1^{(h-1)},\ldots,W_{r_{h-1}}^{(h-1)}$ of $W^{(h-1)}$ to be those irreducible components of $Z^{(1)}=Z_1 \cup \ldots \cup Z_t$ that contain elements from $L_2$. Since by definition of $Z^{(1)}$ we have that $\deg(Z_i) \ge \epsilon_1 \deg(Z_L)$ for every $1 \le i \le t$, and we know that $\deg(Z^{(1)}) \lesssim_{B^{(h)}} \deg(Z_L)$ by (\ref{FF}), we see that
 $$r_{h-1} \le t \lesssim_{B^{(h)}} \epsilon_1^{-1} \lesssim_{B^{(h)}} 1.$$
 Similarly, we deduce that
 $$\deg(W_i^{(h-1)}) \lesssim_{B^{(h)}} \deg(W_j^{(h-1)}),$$
 for every $1 \le i,j \le r_{h-1}$. 
 
 From Bezout's theorem we see that $\deg(P_{n-h+1}) \gtrsim_{B^{(h)}} \deg(P_{n-h})$, since otherwise we would be contradicting Lemma \ref{loweb}. It then follows by construction that we also have $\delta(Z_i) \lesssim_{B^{(h)}} \deg(P_{n-h+1})$ for every $1 \le i \le t$. On the other hand, we know by definition that for every $1 \le i \le t $ we can find some polynomial $f_i$ of degree at most $\delta(Z_i)$ that vanishes on $Z_i$ without vanishing on any of the components of $W^{(h)}$ containing $Z_i$, so by Bezout's theorem it must be $\deg(Z_i) \le \delta(Z_i) \deg(W^{(h)})$. Hence,
 $$ \delta(Z_i) \gtrsim_{B^{(h)}} \frac{\deg(Z_L)}{\deg(W^{(h)})} \gtrsim_{B^{(h)}} \deg(P_{n-h+1}),$$
 where we are using (\ref{owl}). We have thus shown that
 $$ \delta(W_i^{(h-1)}) \sim_{B^{(h)}} \deg(P_{n-h+1}),$$
 for every $1 \le i \le r_{h-1}$.
 
  It only remains to show that 
 \begin{equation}
 \label{al}
  \deg(L_2) \gtrsim_{B^{(h)}} \deg(Z_i) \delta(Z_i)^{h-1-l},
  \end{equation}
 for every $ 1 \le i \le t$, since the corresponding estimates with respect to $W^{(k)}$, $h \le k < n$, follow immediately from the fact that $\deg(L_2) \gtrsim_{B^{(h)}} \deg(L) \gtrsim_{B^{(h)}} \deg(T)$. Similarly, it will suffice to show that (\ref{al}) holds for $L$ in place of $L_2$. Suppose then that $\deg(L) \le \epsilon_3 \deg(Z_i) \delta(Z_i)^{h-1-l}$ for some small $\epsilon_3>0$ to be specified. Since $\deg(Z_i) \le \deg(W^{(h)}) \deg(P_{n-h+1})$, this means that 
 $$\deg(L) \lesssim_{B^{(h)}} \epsilon_3 \deg(W^{(h)}) \deg(P_{n-h+1})^{h-l}.$$
 If $\epsilon_3>0$ is chosen sufficiently small with respect to $B^{(h)}$ this would imply by Lemma \ref{BL} that we can find a polynomial of degree strictly less than $\deg(P_{n-h+1})$ vanishing on $L$ without vanishing on any irreducible component of $W^{(h)}$, contradicting our choice of $P_{n-h+1}$. The result follows.
\end{proof}

\subsection{Relative partitioning lemma}
\label{44}

The next lemma shows how we can iterate Lemma \ref{CII} to obtain a partition of any subset $L \subseteq \tilde{T}$, with a controlled error term for the incidences with respect to a fixed set of varieties $S'$ and such that one element of this partition satisfies a weak form of the dichotomy discussed in \S \ref{OM}. In the next section, we will show how to combine this with the tools we have developed in the rest of this article to give a proof of Theorem \ref{I2}.

\begin{lema}
\label{RP2}
Let $\varepsilon_2,\tau > 0$. Let $L \subseteq \tilde{T} \subseteq T$ and let $S'$ be a finite set of varieties of dimension $\dim(T)-1$. Then, there exist partitions $S'=S_1 \cup \ldots \cup S_r$, $L=L_1 \cup \ldots \cup L_r$, such that
\begin{equation}
\label{recur0}
\mathcal{I}(S',L) \le \sum_{i=1}^r \mathcal{I}(S_i,L_i) + O_{\varepsilon_2}(\deg(L) \deg_R(\tilde{T},W^{(h)})),
 \end{equation}
Furthermore, $\deg(L_r) \gtrsim_{\varepsilon_2,\tau} \deg(L)$ and either $\deg(L_r) < \frac{\varepsilon_2}{2} \deg(T)$ or $L_r$ does not a admit a $\tau$-good partition at level $h$ with respect to $\tilde{T}$.
\end{lema}

\begin{proof}
We are going to proceed recursively, building partitions $S'=S_1 \cup \ldots \cup S_s$, $L=L_1 \cup \ldots \cup L_s$ and then partitioning both $S_s$ and $L_s$ into two new sets that, using a convenient abuse of notation, we rename as $S_s$ and $S_{s+1}$ and $L_s$ and $L_{s+1}$ respectively. So we start with $S_1=S'$, $L_1=L$ and notice that if $L_1$ does not admit a $\tau$-good partition at level $h$ with respect to $\tilde{T}$ or has degree less than $\frac{\varepsilon_2}{2} \deg(T)$ we are done. Hence, we may recursively assume that we have constructed partitions $S'=S_1 \cup \ldots \cup S_s$, $L=L_1 \cup \ldots \cup L_s$, such that
\begin{equation}
\begin{aligned}
\label{recur}
\mathcal{I}(S',L) \le \sum_{i=1}^s \mathcal{I}(S_i,L_i) &+ \left( \log_2 \frac{\deg(L)}{\deg(L_s)} \right) \deg(L) \deg_R(\tilde{T},W^{(h)}) \\
&+ \sum_{i=1}^{s-1}\deg(L_i) \deg_R(\tilde{T},W^{(h)}),
\end{aligned}
 \end{equation}
with $L_s$ having degree at least $\frac{\varepsilon_2}{2} \deg(T)$ and admitting a $\tau$-good partition at level $h$ with respect to $\tilde{T}$. We rename $\tilde{S}=S_s$, $\tilde{L}=L_s$ and let $f_s$ be a polynomial giving a $\tau$-good partition of $\tilde{L}$ at level $h$ with respect to $\tilde{T}$.

Assume first that both elements of this partition, $\tilde{L}_{f_s}$ and $\tilde{L} \setminus \tilde{L}_{f_s}$, have degree at least $\frac{\varepsilon_2}{2} \deg(T)$. If $\deg(\tilde{L}_{f_s}) > \deg(\tilde{L})/2$, we let $L_s = \tilde{L}_{f_s}, S_s = \tilde{S}_{f_s}$ and let $L_{s+1},S_{s+1}$ be their complements in $\tilde{L},\tilde{S}$ respectively. By Lemma \ref{CII}, we have that
\begin{equation}
\label{decom1}
 \mathcal{I}(\tilde{S},\tilde{L}) \le \mathcal{I}(S_s,L_s) + \mathcal{I}(S_{s+1},L_{s+1}) + \deg(L_{s+1}) \deg_R(\tilde{T},W^{(h)}).
 \end{equation}
 We insert this in (\ref{recur}) to obtain the corresponding bound in this case, noticing that 
 $$ \log_2 \frac{\deg(L)}{\deg(\tilde{L})} + 1 \le \log_2  \frac{\deg(L)}{\deg(L_{s+1})},$$
since $\deg(L_{s+1}) \le \deg(\tilde{L})/2$.

On the other hand, if $\deg(\tilde{L}_{f_s}) \le \deg(\tilde{L})/2$, we let $L_{s+1} = \tilde{L}_{f_s}, S_{s+1} = \tilde{S}_{f_s}$ and let $L_{s},S_{s}$ be their complements in $\tilde{L},\tilde{S}$ respectively. In this case, we have by Lemma \ref{CII} that
\begin{equation}
\label{decom2}
 \mathcal{I}(\tilde{S},\tilde{L}) \le \mathcal{I}(S_s,L_s) + \mathcal{I}(S_{s+1},L_{s+1}) + \deg(L_s) \deg_R(\tilde{T},W^{(h)}) ,
 \end{equation}
and we may insert this in (\ref{recur}) as we did before.

In either case, if $L_{s+1}$ does not admit a $\tau$-good partition at level $h$ with respect to $\tilde{T}$, then the result is proven since
$$ \log_2  \frac{\deg(L)}{\deg(L_{s+1})} \lesssim_{\varepsilon_2} 1,$$
 while otherwise we have completed the recursive step. Notice that since $\deg(L_{s+1})$ decreases at each step while remaining bounded from below by $\frac{\varepsilon_2}{2} \deg(T)$ this situation can only occur finitely many times.

It remains to show what happens when at least one element of the partition of $\tilde{L}$ has degree less than $\frac{\varepsilon_2}{2} \deg(T)$. In this case, we write $L_{s+1} = \tilde{L}_{f_s}, S_{s+1} = \tilde{S}_{f_s}$ and let $L_{s},S_{s}$ be their complements in $\tilde{L},\tilde{S}$ respectively. If only $L_s$ has degree less than $\frac{\varepsilon_2}{2} \deg(T)$, then we insert the bound (\ref{decom2}) in (\ref{recur}), which gives a satisfactory bound when we add the last term of both expressions. If $L_{s+1}$ does not admit a $\tau$-good partition at level $h$ with respect to $\tilde{T}$ we obtain the claim of the lemma we are trying to prove and otherwise we have completed the recursive step. It is clear this situation can also only occur finitely many times.

We conclude that after finitely many steps of the recursive process, either the result is proven or we must reach a first instance where $\deg(L_{s+1}) < \frac{\varepsilon_2}{2} \deg(T)$. But then, we may insert (\ref{decom2}) in (\ref{recur}) as before and the result follows, since 
$$\deg(L_{s+1}) > \tau \deg(\tilde{L}) \ge \tau \frac{\varepsilon_2}{2} \deg(T),$$
by the definition of a $\tau$-good partition and our lower bound on $\tilde{L}$.
\end{proof}
 
 \section{Proof of Theorem \ref{I2}}
 \label{5}
 
We now turn to the proof of Theorem \ref{I2} and therefore of Theorem \ref{I0}. As we discussed in \S \ref{31}, we may assume $\dim(T) < n$, since the result is trivial otherwise. We have also seen the result is clear if $|T| = O_n(1)$ or $T$ reduces strongly to level $l$, so we may assume $T$ reduces strongly to some level $l <h \le n$ and that the result holds for any proper subset of $T$ and if $T$ reduces strongly to some level $h'<h$. 

We will use the notations of Definition \ref{reduction} and Definition \ref{SR} to refer to the corresponding objects associated with $T$. We may abbreviate $\Dc_k(T)$ as $\Dc_k$ and we shall also say that $T$ admits a $\tau$-good partition at level $k$ to mean that it admits a $\tau$-good partition over $W^{(k)}$.

By Lemma \ref{Dta}, Lemma \ref{GSiegel0}, Lemma \ref{simplelinear} and Lemma \ref{q} we know there exists a polynomial $P$ of degree 
$$\lesssim_{B^{(h)}} \left( \frac{\deg(S)}{\Delta_{n-s}(W_1^{(h)})} \right)^{\frac{1}{s-\dim(S)}},$$
for some $h \le s \le n$, vanishing on $S$ without vanishing on $W_i^{(h)}$ for any $1 \le i \le r_h$. We will use this polynomial to bound the relative degree $\deg_R(\tilde{T},W^{(h)})$, for an adequate subset $\tilde{T}$ of $T$.

Let us assume first that $\deg(T_P) < \deg(T)/2$ and notice that
 \begin{equation}
 \begin{aligned}
  \mathcal{I}(S,T) &\le \mathcal{I}(S,T \setminus T_P) + \mathcal{I}(S,T_P).
  \end{aligned}
  \end{equation}
  Since $T_P$ is then a proper subset of $T$, we know by induction that
    \begin{equation}
  \begin{aligned}
  \label{sun0}
  \mathcal{I}(S,T_P) \le& \sum_{m=1}^{n-\dim(S)} K_{l,m} \deg(S)^{1/m} \deg(T_P)^{1-1/m} \Dc_{m+\dim(S)}(T_P)^{1/m} \\
  \le& \, K_{l,1} \deg(S) +  \sum_{m=2}^{n-\dim(S)} \frac{K_{l,m}}{\sqrt{2}} \deg(S)^{1/m} \deg(T)^{1-1/m} \Dc_{m+\dim(S)}(T)^{1/m}.
  \end{aligned}
  \end{equation}
  On the other hand, by Bezout's theorem, we see that
  \begin{equation}
  \begin{aligned} 
   \label{sun}
   \mathcal{I}(S,T \setminus T_P) &\le \deg(T \setminus T_P) \deg(P) \\
   &\lesssim_{B^{(h)}} \deg(T \setminus T_P)  \left( \frac{\deg(S)}{\Delta_{n-s}(W_1^{(h)})} \right)^{\frac{1}{s-\dim(S)}} \\
   &\lesssim_{B^{(h)}} \deg(T) \left( \frac{\deg(S)}{\deg(T)} \right)^{\frac{1}{s-\dim(S)}} \Dc_{s}(T)^{\frac{1}{s-\dim(S)}}.
   \end{aligned}
   \end{equation}
   Here we are using the fact that by Lemma \ref{Dta}, Definition \ref{reduction} and Lemma \ref{q} it is
   \begin{equation}
   \begin{aligned}
   \label{hello}
    \Delta_{n-s}(W_1^{(h)}) \sim_{B^{(h)}} \prod_{i=1}^{n-s} \delta_i(W_1^{(h)}) 
    \sim_{B^{(h)}} \prod_{i=1}^{n-s} \deg(P_i) 
    \sim_{B^{(h)}} \deg(W^{(s)}).
    \end{aligned}
    \end{equation}
   The result then follows in this case upon summing (\ref{sun}) to (\ref{sun0}) and choosing the constants $K_{l,m}$ to satisfy
   $$ \frac{K_{l,m}}{\sqrt{2}} + O_{B^{(h)}}(1) \le K_{l,h-\dim(S)} \le K_{l,m},$$
   for every $h - \dim(S) \le m \le n - \dim(S)$.
   
   We may therefore assume from now on that $\deg(T_P) \ge \deg(T)/2$ and write $\tilde{T}=T_P$. In particular, we see that $\tilde{T}$ reduces to level $h$ by Lemma \ref{srl}. Let $f$ be the polynomial given in Lemma \ref{T0} with respect to $\tilde{T}$. We have by Lemma \ref{CII} and Lemma \ref{or} that
$$  \mathcal{I}(S,\tilde{T}) \le \mathcal{I}(S \setminus S_f, \tilde{T} \setminus \tilde{T}_f) + \mathcal{I}(S_f, \tilde{T}_f) + \deg(\tilde{T}) \deg_R(\tilde{T},W^{(h)}).$$
We write $T_0 = (\tilde{T})_f$, $S_0 = S_f$ and apply Lemma \ref{RP2} with $L=\tilde{T} \setminus T_0$, $S'=S \setminus S_0$, $\varepsilon_2 \gtrsim_{B^{(h)}} 1$ as in Corollary \ref{DH2} and $\tau > 0$ of the form provided in Proposition \ref{nogood} with respect to some small $\varepsilon_3 \sim 1$. This produces partitions $\tilde{T}= T_0 \cup T_1 \cup \ldots \cup T_r$ and $S=S_0 \cup S_1 \cup \ldots \cup S_r$ with
 $$ \mathcal{I}(S,\tilde{T}) \le \sum_{i=0}^r \mathcal{I}(S_i,T_i) + O_{B^{(h)}} \left( \deg(\tilde{T}) \deg_R(\tilde{T},W^{(h)}) \right),$$
and with $\deg(T_r) \gtrsim_{B^{(h)}, a_h^{(h-1)}} \deg(T)$. If $T_r$ does not admit a $\tau$-good partition at level $h$ with respect to $\tilde{T}$ we can obtain from Proposition \ref{nogood} a polynomial $f^{\ast}$ of degree $\lesssim_{B^{(h)}} \deg_R(\tilde{T},W^{(h)})$ with
\begin{equation}
\begin{aligned}
 \mathcal{I}(S_r,T_r) \le \mathcal{I}(S_r \setminus (S_r)_{f^{\ast}}, T_r \setminus (T_r)_{f^{\ast}}) &+ \mathcal{I}((S_r)_{f^{\ast}}, (T_r)_{f^{\ast}}) \\
 &+ O_{B^{(h)}} \left( \deg(T_r) \deg_R(\tilde{T},W^{(h)}) \right),
 \end{aligned}
 \end{equation}
by Lemma \ref{CII}, and such that $T_r \setminus (T_r)_{f^{\ast}}$ reduces strongly to level $h-1$ while having size $\gtrsim_{B^{(h)},a_h^{(h-1)}} \deg(\tilde{T})$. We write $v=r+1$ in this case and $v=r$ if $T_r$ does admit a $\tau$-good partition at level $h$ with respect to $\tilde{T}$ with $\tau$ as above. In the former case, we then relabel $(T_r)_{f^{\ast}}$ as $T_r$, write $T_{r+1}$ for $T_r \setminus (T_r)_{f^{\ast}}$ and we similarly define $S_r$ and $S_{r+1}$. 

We have therefore partitioned $S$ and $\tilde{T}$ into $v+1$ pieces with
  \begin{equation}
  \label{v}
   \mathcal{I}(S,\tilde{T}) \le \sum_{i=0}^v \mathcal{I}(S_i,T_i) + O_{B^{(h)}} \left( \deg(\tilde{T}) \deg_R(\tilde{T},W^{(h)}) \right),
   \end{equation}
  with $\deg(T_v) \gtrsim_{B^{(h)}, a_h^{(h-1)}} \deg(T)$ and such either $\deg(T_v) < \frac{\varepsilon_2}{2} \deg(T)$ or $T_v$ reduces strongly to level $h-1$. Applying induction for each $0 \le i \le v-1$, we conclude that
  \begin{equation}
  \begin{aligned}
  \label{kd}
  \mathcal{I}(S,\tilde{T}) \le \sum_{i=0}^{v-1}& \sum_{1 \le m \le n-\dim(S)} K_{l,m} \deg(S_i)^{1/m} \deg(T_i)^{1-1/m} \Dc_{m+\dim(S)}(T_i)^{1/m} \\
  &+  \mathcal{I}(S_v,T_v) + O_{B^{(h)}} \left( \deg(\tilde{T}) \deg_R(\tilde{T},W^{(h)}) \right).
  \end{aligned}
  \end{equation}
Assume $\deg(S_v) \le \eta \deg(S)$ for some sufficiently small $\eta \gtrsim_{B^{(h)},a_h^{(h-1)}} 1$. This necessarily means that the partition of $T$ we have constructed is nontrivial and in particular $\deg(T_v)<\deg(T)$. We can therefore apply induction to conclude that
\begin{equation}
\begin{aligned}
\label{vs}
\mathcal{I}(S_v,T_v) &\le \sum_{1 \le m \le n-\dim(S)} K_{l,m} \deg(S_v)^{1/m} \deg(T_v)^{1-1/m} \Dc_{m+\dim(S)}(T_v)^{1/m} \\
&\le \eta^{\frac{1}{n-\dim(S)}}  \sum_{1 \le m \le n-\dim(S)} K_{l,m} \deg(S)^{1/m} \deg(T)^{1-1/m} \Dc_{m+\dim(S)}(T)^{1/m}.
\end{aligned}
\end{equation}
On the other hand, since $\deg(T_v) \gtrsim_{B^{(h)}, a_h^{(h-1)}} \deg(T)$, we see from Hölder's inequality that we can find some $\epsilon_1 \gtrsim_{B^{(h)}, a_h^{(h-1)}} 1$ such that (\ref{kd}) is bounded by
\begin{equation}
\begin{aligned}
\label{alt}
(1-\epsilon_1) &\sum_{2 \le m \le n-\dim(S)} K_{l,m} \deg(S)^{1/m} \deg(T)^{1-1/m} \Dc_{m+\dim(S)}(T)^{1/m} \\
&+ K_{l,1} \deg(S \setminus S_v) + \mathcal{I}(S_v,T_v) + O_{B^{(h)}} \left( \deg(\tilde{T}) \deg_R(\tilde{T},W^{(h)}) \right).
\end{aligned}
\end{equation}
We now use the fact that
\begin{equation}
\begin{aligned}
\label{sm}
  \deg_R(\tilde{T},W^{(h)}) \le \deg(P) &\lesssim_{B^{(h)}} \left( \frac{\deg(S)}{\Delta_{n-s}(W_1^{(h)})} \right)^{\frac{1}{s-\dim(S)}},\\
  &\lesssim_{B^{(h)}} \left( \frac{\deg(S)}{\deg(T)} \Dc_{s}(T) \right)^{\frac{1}{s-\dim(S)}}.
  \end{aligned}
  \end{equation}
  where the last bound follows the same argument we used when dealing with the case $\deg(T_P) < \deg(T)/2$. We now insert (\ref{vs}) and (\ref{sm}) in (\ref{alt}) and add the resulting bound for $\mathcal{I}(S,\tilde{T})$ to the bound (\ref{sun}) for $\mathcal{I}(S,T \setminus \tilde{T})$. The result then follows in this case upon choosing $\eta$ sufficiently small with respect to $\epsilon_1$ so that $\eta^{\frac{1}{n-\dim(S)}}<\epsilon_1/2$, say, and the constants $K_{l,m}$ so that
  $$ K_{l,m} (1-\epsilon_1/2) + O_{B^{(h)}}(1) \le K_{l,h-\dim(S)} \le K_{l,m},$$
  for every $h-\dim(S) \le m \le n-\dim(S)$.
    
  We may therefore assume from now on that $\deg(S_v) \gtrsim_{B^{(h)}, a_h^{(h-1)}} \deg(S)$. Using again induction for each $0 \le i \le v-1$, Hölder's inequality, the estimates (\ref{v}) and (\ref{sm}) and the bound (\ref{sun}) for $\mathcal{I}(S,T \setminus \tilde{T})$, we see that $\mathcal{I}(S,T)$ is bounded by
\begin{equation}
\begin{aligned}
\label{pas}
\sum_{1 \le m \le n-\dim(S)}& K_{l,m} \deg(S \setminus S_v)^{1/m} \deg(T \setminus T_v)^{1-1/m} \Dc_{m+\dim(S)}(T)^{1/m} \\
&+ \mathcal{I}(S_v,T_v) + O_{B^{(h)}} \left( \deg(T)\left( \frac{\deg(S)}{\deg(T)} \Dc_{s}(T) \right)^{\frac{1}{s-\dim(S)}} \right).
\end{aligned}
\end{equation}

If $T_v$ reduces strongly to level $h-1$ we have by induction
\begin{equation}
\label{hun}
 \mathcal{I}(S_v,T_v) \le \sum_{1 \le m \le n-\dim(S)} K_{l,m}^{(h-1)} \deg(S_v)^{1/m} \deg(T_v)^{1-1/m} \Dc_{m+\dim(S)}(T)^{1/m}.
 \end{equation}
Since $\deg(S) \lesssim_{B^{(h)}, a_h^{(h-1)}} \deg(S_v)$ and $\deg(T) \lesssim_{B^{(h)}, a_h^{(h-1)}} \deg(T_v)$, the result follows in this case upon inserting (\ref{hun}) in (\ref{pas}) and choosing the constants $K_{l,m}$ to satisfy
$$ K_{l,h-1-\dim(S)} + O_{B^{(h)}, a_h^{(h-1)}}(1) < K_{l,h-\dim(S)}.$$
We may therefore assume that $\deg(T_v) < \frac{\varepsilon_2}{2} \deg(T)$ in which case we can apply Corollary \ref{DH2} to deduce that $\Dc_k(T_v) < \Dc_k/2$ for every $h \le k \le n$. As a consequence, applying induction to estimate $\mathcal{I}(S_v,T_v)$ in (\ref{pas}) and using again that $\deg(S) \lesssim_{B^{(h)}, a_h^{(h-1)}} \deg(S_v)$ and $\deg(T) \lesssim_{B^{(h)}, a_h^{(h-1)}} \deg(T_v)$, the result also follows in this case upon choosing the constants $K_{l,m}$ to satisfy
$$ \frac{K_{l,m}}{2^{1/m}} + O_{B^{(h)}, a_h^{(h-1)}}(1) < K_{l,h-\dim(S)} \le K_{l,m},$$
  for every $h-\dim(S) \le m \le n-\dim(S)$.

\section{Alternative bounds on the relative degree}
\label{6}

In this section we show how to establish Theorem \ref{R} and Theorem \ref{p}. This is accomplished by means of slight modifications of the proof of Theorem \ref{I0} based on alternative ways of bounding the relative degrees.

\subsection{Proof of Theorem \ref{R}}

We now proceed to the proof of Theorem \ref{R}. We will use the following standard lemma, a proof of which can be found in \cite{W5} (see also \cite{KST}).

\begin{lema}
\label{trivialbound}
If $S$ is $k$-free with respect to $T$, then
$$ \mathcal{I}(S,T) \le 2 |S||T|^{1-1/k}+ (k-1) |T|.$$
\end{lema}

As with Theorem \ref{I0}, we shall deduce Theorem \ref{R} from the following equivalent statement.

\begin{teo}
\label{Rh}
For every integer $n \ge 1$ there exist constants $K_0 < K_1 < \ldots < K_n \lesssim_{n} 1$ such that the following holds. Let $T$ be a family of irreducible algebraic curves in $\R^n$ that reduces strongly to level $h$. Then, for any set of points $S \subseteq \R^n$ that is $k$-free with respect to $T$, we have the bound
$$ \mathcal{I}(S,T) \le \sum_{0 \le m \le n} K_m^{(h)} k^{1-\alpha(k,m)} |S|^{\alpha(k,m)} \deg(T)^{1-\alpha(k,m)} \Dc_m (T)^{\frac{k-1}{k} \alpha(k,m)},$$
with $\alpha(k,m)$ as in Theorem \ref{R} and where $K_m^{(h)}=K_h$ if $m \ge h$ and $K_m^{(h)} = K_m$ otherwise.
\end{teo} 

\begin{proof}
Given integers $k,s \ge 1$, we shall write
$$ \alpha(k,s) = \frac{k}{s(k-1)+1} , \, \, \beta(k,s)=\frac{s(k-1)}{s(k-1)+1} , \, \, \gamma(k,s) = \frac{k-1}{s(k-1)+1}.$$
Notice that $1-\alpha(k,s)=\beta(k,s)-\gamma(k,s)$.

As in the proof of Theorem \ref{I2} we may assume that $T$ reduces strongly to some level $1 < h \le n$ and that the result holds both for proper subsets of $T$ and if $T$ reduces strongly to level $h-1$. We shall use the notations in Definition \ref{reduction} and Definition \ref{SR} to refer to the corresponding objects associated with $T$.

Let us consider first a single component $W_1^{(h)}$ of $W^{(h)}$. For every $h \le s \le n$ we shall consider the parameters
$$ M_{n-s} = \left( \frac{|S|^k}{k^k \deg(T) (\prod_{i=0}^{n-s} \delta_i(W_1^{(h)}))^{k-1}} \right)^{\frac{1}{s(k-1)+1}}.$$
Assume $h \le s < n$ is such that $M_{n-s} \lesssim_n \delta_{n-s}(W_1^{(h)})$. It follows that
\begin{equation}
\label{nag}
 M_{n-s-1} = M_{n-s}^{\frac{s(k-1)+1}{s(k-1)+1+(k-1)}} \delta_{n-s}(W_1^{(h)})^{\frac{k-1}{s(k-1)+1+(k-1)}} \lesssim_n \delta_{n-s}(W_1^{(h)}).
 \end{equation}
We now notice that if $M_0 \lesssim_n \delta_0(W_1^{(h)})=1$, it follows that $|S| \lesssim_n k \deg(T)^{1/k}$ and therefore we deduce from Lemma \ref{trivialbound} that
$$ \mathcal{I}(S,T) \lesssim_n k \deg(T).$$
We may therefore assume from now on that this is not the case. Combining this with the observation (\ref{nag}), we see that there must be a largest choice of $h \le s \le n$ with 
\begin{equation}
\label{EDP0}
\delta_{n-s}(W_1^{(h)}) \lesssim_n M_{n-s} \lesssim_n \delta_{n-s+1}(W_1^{(h)}),
\end{equation}
This allows us to apply Theorem \ref{30par} to find a polynomial $P_1$ of degree $\lesssim_n M_{n-s}$, not vanishing identically on $W_1^{(h)}$ and such that each connected component of $\R^n \setminus Z(P_1)$ contains at most
$$ \lesssim_n \frac{|S|}{M_{n-s}^{s} \prod_{i=0}^{n-s} \delta_i(W_1^{(h)})},$$
elements from $S_{W_1^{(h)}}$. Here we are using Lemma \ref{Dta}.

Let us write $\Omega_1,\ldots,\Omega_r$ for the connected components of $\R^n \setminus Z(P_1)$, $S^{(i)}$ for the elements of $S_{W_1^{(h)}}$ inside of $\Omega_i$ and $T^{(i)}$ for those elements of $T$ intersecting $\Omega_i$. Notice that by Theorem \ref{T5} each $t \in T$ belongs to $T^{(i)}$ for $\lesssim_n \deg(t) \deg(P_1)$ values of $1 \le i \le r$. Using Lemma \ref{trivialbound}, we can therefore bound $\mathcal{I}(S_{W_1^{(h)}} \setminus S_{P_1},T)$ by
\begin{equation}
\begin{aligned}
\label{mj}
&\lesssim \sum_{i=1}^r |S^{(i)}| |T^{(i)}|^{1-1/k} + k|T^{(i)}| \\
 &\lesssim_n \left( \frac{|S|}{M_{n-s}^{s} \prod_{i=0}^{n-s} \delta_i(W_1^{(h)})} \right)^{1-1/k} |S|^{1/k} (\deg(T)M_{n-s})^{1-1/k} + k\deg(T)M_{n-s} \\
 &\lesssim_n k^{1-\alpha(k,s)} |S|^{\alpha(k,s)}\deg(T)^{\beta(k,s)} \left(\prod_{i=0}^{n-s} \delta_i(W_1^{(h)}) \right)^{-\gamma(k,s)} \\
 &\lesssim_n k^{1-\alpha(k,s)} |S|^{\alpha(k,s)} \deg(T)^{1-\alpha(k,s)} \mathcal{D}_{s}(T)^{\frac{k-1}{k}\alpha(k,s)},
 \end{aligned}
 \end{equation}
  where we are using (\ref{hello}).
  
  We may repeat the above argument for each other component $W_i^{(h)}$ of $W^{(h)}$, $2 \le i \le r_h$, obtaining $r_h=O_{B^{(h)}}(1)$ bounds of the form (\ref{mj}) and corresponding polynomials $P_i$ with $\deg(P_i) \lesssim_{B^{(h)}} M_{n-s}$, for the same choice of $s$ as above (by Lemma \ref{q}). We therefore see that, writing $Z = \bigcup_{i=1}^{r_h} Z(P_i) \cap W_i^{(h)}$, the result will follow from an adequate bound on $\mathcal{I}(S_Z,T)$.
  
  Let $h \le g \le n$ be such that
  \begin{equation}
  \label{nam}
  \eta_{n-g} \delta_{n-g}(W_1^{(h)}) \le M_{n-s}< \eta_{n-g+1} \delta_{n-g+1}(W_1^{(h)}),
  \end{equation}
  where the constants $\eta_0,\ldots,\eta_{n-h+1} \gtrsim_{B^{(h)}} 1$ are chosen so that $\eta_{i+1}$ is sufficiently small with respect to $\eta_i$ and $B^{(h)}$. Notice that since $\delta_{n-h+1}(W_1^{(h)}) = \infty$, such a choice always exists. Using the notation of Definition \ref{reduction}, for each $1 \le i \le r_h$, we let $W_{j_i}^{(g)}$ be an irreducible component of $W^{(g)}$ containing $W_i^{(h)}$. Notice also that by Lemma \ref{q} and (\ref{nam}) we have
  \begin{equation}
  \label{EDP}
  \delta(W^{(g)}) \sim_{B^{(h)}} \deg(P_{n-g}) \sim_{B^{(h)}}  \delta_{n-g}(W_1^{(h)}) \lesssim_{B^{(h)},\eta_{n-g}} M_{n-s}.
  \end{equation}
We now observe that $Z^{(g)} = \bigcup_{i=1}^{r_h} Z(P_i) \cap W_{j_i}^{(g)}$ is a $(g-1)$-dimensional algebraic set of degree $\lesssim_{B^{(h)}} M_{n-s}  \deg(W^{(g)})$ containing $Z$. By Lemma \ref{BL} and (\ref{EDP}), we can then find a polynomial $P$ of degree $\lesssim_{B^{(h)},\eta_{n-g}} M_{n-s}$, vanishing on $Z^{(g)}$, and therefore on $Z$, without vanishing identically on any irreducible component of $W^{(g)}$. But if $P$ vanishes on $W_i^{(h)}$, for some $1 \le i \le r_h$, this implies by Lemma \ref{q} that 
$$\delta_{n-g+1}(W_1^{(h)}) \sim_{B^{(h)}} \delta_{n-g+1}(W_i^{(h)}) \lesssim_{B^{(h)},\eta_{n-g}} M_{n-s},$$
which contradicts (\ref{nam}) if $\eta_{n-g+1}$ was chosen sufficiently small with respect to $B^{(h)}$ and $\eta_{n-g}$. We conclude that $P$ vanishes on $Z^{(g)}$, and therefore on $S_Z$, without vanishing identically on any irreducible component of $W^{(h)}$.

From here the proof proceeds by exactly the same argument as in the proof of Theorem \ref{I2}, using that in our current case the bound (\ref{sm}) can now be replaced by
  \begin{equation}
  \begin{aligned}
  \label{agg}
  \deg_R(T_P,W^{(h)}) \le \deg(P) &\lesssim_{B^{(h)}} M_{n-s} \\
  &\lesssim_{B^{(h)}} \left( \frac{|S|}{\deg(T)} \right)^{\alpha(k,s)} \Dc_s(T)^{\frac{k-1}{k}\alpha(k,s)},
  \end{aligned}
  \end{equation}
  where we are using once again (\ref{hello}). Notice that writing $\tilde{T}=T_P$, the terms (\ref{mj}) and $O_{B^{(h)}} \left( \deg(\tilde{T}) \deg_R(\tilde{T},W^{(h)}) \right)$ will then be absorbed in the same way the corresponding terms where absorbed during the proof of Theorem \ref{I2}. We remark that the only slight difference in the proof is that in the current case, when evaluating inductively the expressions $\mathcal{I}(S_i,T_i)$ that arise in the argument, we will have an additional term of the form $K_0 k \deg(T_i)$. Nevertheless, these terms will clearly add up to at most $K_0 k \deg(T)$ by the disjointness of the sets $T_i$ constructed during the proof.
  \end{proof}
  
  \subsection{Proof of Theorem \ref{p}}
  
  To prove Theorem \ref{p} we will use the fact that if $T$ reduces to level $h$ and $p$ is a $\delta$-property with respect to an irreducible component of $W^{(h)}$, then we can essentially use the implicit polynomial in this definition to bound the relative degree of $T$ with respect to $W^{(h)}$. Since the argument is almost the same as in the previous results, we will only explain the differences in the proof.
  
 As with the previous theorems, we establish the variants given by Theorem \ref{I2} or Theorem \ref{Rh} with the quantities $\Dc_m(T)$ replaced by $\Dc_m^p(T)$, but with the additional difference that, when $l < h < n$, we now also allow an extra term of the form $O_{B^{(l)},h}(\deg(T) \deg_R(T,W^{(h+1)}))$ in the bound. Since every $T$ reduces strongly to level $n$, this clearly implies Theorem \ref{p}. 
 
 The case $l=h$ is trivial as before. Let us assume $l < h \le n$, that $T$ reduces strongly to level $h$ and that the result has been established for proper subsets of $T$ and if $T$ reduces strongly to level $h-1$.
 
Write $W_1^{(h)}, \ldots, W_q^{(h)}$ for those components of $W^{(h)}$ that are not in $\mathcal{C}_p$ and $W_{q+1}^{(h)}, \ldots, W_{r_h}^{(h)}$ for the remaining ones. Let us also write $W^{\ast}= W_{q+1}^{(h)} \cup \ldots \cup W_{r_h}^{(h)}$. 

We assume first that $\deg(T_{W^{\ast}}) \gtrsim_{B^{(h)}} \deg(T)$. Notice that if $\deg(S) \lesssim_n \deg(W^{(h)}) \delta(W^{(h)})^{h-\dim(S)}$ we can use Lemma \ref{BL} to find a polynomial $P$ of degree 
 $$\lesssim_n \delta(W^{(h)}) \lesssim_{B^{(h)}} \deg_R(T,W^{(h+1)}),$$
 vanishing on $S$ without vanishing identically on any irreducible component of $W^{(h)}$. Otherwise, we know by Lemma \ref{BL} that we can find a polynomial $P$ vanishing on $S$ without vanishing identically on any irreducible component of $W^{(h)}$ with $P$ having degree
\begin{equation}
\begin{aligned}
 \lesssim_{B^{(h)}} \left( \frac{\deg(S)}{\deg(W^{(h)})} \right)^{\frac{1}{h - \dim(S)}} &\lesssim_{B^{(h)}} \left( \frac{\deg(S)}{\deg(T)} \frac{\deg(T_{W^{\ast}})}{\deg(W^{\ast})} \right)^{\frac{1}{h - \dim(S)}}, \\
 &\lesssim_{B^{(h)}} \left( \frac{\deg(S)}{\deg(T)} \Dc_h^p(T) \right)^{\frac{1}{h - \dim(S)}}.
 \end{aligned}
 \end{equation}
 To establish the analogue of Theorem \ref{I2} we can then proceed as in the proof of that result to reduce the problem to the estimation of $\mathcal{I}(S,\tilde{T})$, where $\tilde{T} = T_P$ satisfies $\deg(\tilde{T}) \ge \deg(T)/2$ and 
\begin{equation}
\label{yy}
\deg_R(\tilde{T},W^{(h)}) \le \deg(P).
\end{equation}
Because of the bounds we have placed on $\deg(P)$ we can proceed exactly as in the proof of Theorem \ref{I2} to obtain a satisfactory estimate. Notice that the only difference is that if we have to deal with a component $T_r$ of the partition that reduces strongly to level $h-1$, this will give rise to an additional term of the form $O_{B^{(l)},h-1}(\deg(T_r) \deg_R(T_r,W^{(h)}))$ if $h \ge l+2$. However, this term is not problematic since this expression is bounded by
$$ \lesssim_{B^{(l)},h} \deg(T) \deg(P),$$
where we are using (\ref{yy}). Our bounds on $\deg(P)$ make this an acceptable term and this completes the analysis in this case.
 
 We may therefore assume from now on that $\deg(T_{W^{\ast}}) \le \epsilon \deg(T)$ for some small $\epsilon \gtrsim_{B^{(h)}} 1$. Let $r$ be as in the definition of a $p$-set. Upon discarding a set $S' \subseteq S$ contributing $O_{B^{(h)}} (\deg(S'))$ incidences, we may assume that all incidences involve an element of $\mathcal{P}_r(T_{W_i^{(h)}})$ for some $1 \le i \le r_h$. By the definition of being a $\delta$-property, Lemma \ref{q}, Bezout's theorem and Lemma \ref{BL}, we see that there exists some polynomial $f$ of degree $O_n(\delta(W^{(h)})) \sim_{B^{(h)}} \deg_R (T, W^{(h+1)})$ vanishing on $\bigcup_{j=1}^{q} \mathcal{P}_r(T_{W_j^{(h)}})$ without vanishing identically on $W_i^{(h)}$ for any $1 \le i \le r_h$.
 
 Consider the set of elements $S_{i,j}$ of $S \setminus S_f$ that lie in $W_i^{(h)} \cap W_j^{(h)}$ for some pair $1 \le i \le q < j \le r_h$. By definition of an entangled pair, we see there exists some polynomial of degree $\lesssim_n \delta(W^{(h)})$ vanishing on $S_{i,j}$ without vanishing identically on one of these two varieties. In particular, this gives us an algebraic set of dimension $h-1$ and degree $\lesssim_n \delta(W^{(h)}) \deg(W^{(h)})$ containing $S_{i,j}$. Applying Lemma \ref{BL} we can then conclude that there exists some polynomial $g$ of degree $\lesssim_{B^{(h)}} \delta(W^{(h)})$ vanishing on all sets $S_{i,j}$ of the form above without vanishing identically on any irreducible component of $W^{(h)}$.
 
Notice that, by construction, an element of $S \setminus S_{fg}$ can only be incident to an element of $T_{W^{\ast}}$. Since
\begin{equation}
\label{crisp}
\deg(f) + \deg(g) \lesssim_{B^{(h)}} \delta(W^{(h)}) \lesssim_{B^{(h)}} \deg_R (T,W^{(h+1)}),
\end{equation}
we have by Lemma \ref{CII} the estimate
 $$ \mathcal{I}(S,T) \le \mathcal{I}(S_{fg},T_{fg}) + \mathcal{I}(S \setminus S_{fg},  T_{W^{\ast}} \setminus T_{fg}) + O_{B^{(h)}}( \deg(T) \deg_R(T,W^{(h+1)}) ).$$
 We evaluate the second term by induction. Writing $P=fg$, thanks to our upper bound on $\deg(T_{W^{\ast}})$ we can proceed as in the proof of Theorem \ref{I2} to reduce the problem to the estimation of $\mathcal{I}(S_P,\tilde{T})$ with $\tilde{T}=T_P$ and $\deg(\tilde{T}) \ge \deg(T)/2$. Because $\deg(P)$ is bounded by (\ref{crisp}), we can use the same argument as in the proof of that result to obtain a satisfactory bound.
 
 For the analogue of Theorem \ref{Rh} we may again start by discarding a subset $S' \subseteq S$ to reduce to the case $S= \bigcup_{i=1}^{r_h} \mathcal{P}_r(T_{W_i^{(h)}})$. As in the proof of Theorem \ref{Rh}, we may now find some $h \le s \le n$ such that $M_{n-s}$ satisfies (\ref{EDP0}). It is easy to ensure that in this situation we have $\delta_{n-s}(W_1^{(h)}) < \eta \delta_{n-s+1} (W_1^{(h)})$ for some $\eta \gtrsim_{B^{(h)}} 1$ that is sufficiently small with respect to $B^{(h)}$, similarly as it was done in (\ref{nam}).
 
 We now separate the components of $W^{(s)}$ according to whether they belong to $\mathcal{C}_p$, as we did for $W^{(h)}$ in the proof of the analogue of Theorem \ref{I2}. In particular, we write $W^{\ast}$ for the union of those components that belong to $\mathcal{C}_p$. If $\deg(T_{W^{\ast}}) \lesssim_{B^{(h)}} \deg(T)$ we can proceed exactly as we did before in this case to obtain a satisfactory bound, using the fact that $\deg_R(T,W^{(s+1)}) \lesssim_{B^{(h)}} \deg_R(T,W^{(h+1)})$ by Lemma \ref{or} and that a polynomial of degree $\lesssim_{B^{(h)}} \delta(W^{(s)})$ that does not vanish identically on any irreducible component of $W^{(s)}$ will necessarily not vanish identically on any irreducible component of $W^{(h)}$. The latter observation follows from the fact that
 \begin{equation}
 \label{itai}
 \lesssim_{B^{(h)}} \delta(W^{(s)}) \lesssim_{B^{(h)}} \delta_{n-s}(W_i^{(h)}) < \delta_{n-s+1}(W_i^{(h)}),
 \end{equation}
for every $1 \le i \le r_h$, by Lemma \ref{q} and our choice of $\eta$.

We may therefore assume that  $\deg(T_{W^{\ast}}) \gtrsim_{B^{(h)}} \deg(T)$. But then the result follows by the same argument as in the proof of Theorem \ref{Rh}, noticing that by (\ref{hello}) we now have the inequality
$$ \frac{|S|}{\prod_{i=0}^{n-s} \delta_i(W_1^{(h)})} \lesssim_{B^{(h)}} \frac{|S|}{\deg(T)} \frac{\deg(T_{W^{\ast}})}{\deg(W^{\ast})} \lesssim_{B^{(h)}} \frac{|S|}{\deg(T)} \Dc_s^p (T),$$
allowing us to get a bound in terms of $\Dc_m^p(T)$ instead of $\Dc_m(T)$. This gives the desired result.

\end{document}